\documentclass[11pt,leqno]{amsart}

\usepackage[a4paper,lmargin=2cm,rmargin=2cm,tmargin=4cm,bmargin=4cm]{geometry}

\usepackage[centertags]{amsmath}
\usepackage{amsfonts}
\usepackage{amssymb}
\usepackage{amsthm}
\usepackage{dsfont}

\numberwithin{equation}{section}

\usepackage[dvipsnames]{xcolor}

\usepackage{hyperref}
	\hypersetup{breaklinks=true,colorlinks=true,
linkcolor=MidnightBlue,citecolor=MidnightBlue,
urlcolor=MidnightBlue}

\usepackage[normalem]{ulem}
\usepackage{bbm}

\usepackage[english]{babel}
\usepackage[utf8]{inputenc}
\usepackage{enumerate}
\usepackage{mathrsfs}

\theoremstyle{plain}
\newtheorem{theorem}{Theorem}[section]
\newtheorem{corollary}[theorem]{Corollary}
\newtheorem{proposition}[theorem]{Proposition}
\newtheorem{lemma}[theorem]{Lemma}

\theoremstyle{definition}

\newtheorem{fact}[theorem]{Fact}
\newtheorem{example}[theorem]{Example}
\newtheorem{remark}[theorem]{Remark}

\newcommand{\N}{\ensuremath{\mathbb{N}}}
\newcommand{\C}{\ensuremath{\mathbb{C}}}
\newcommand{\K}{\ensuremath{\mathbb{K}}}
\newcommand{\R}{\ensuremath{\mathbb{R}}}
\newcommand{\T}{\ensuremath{\mathbb{T}}}
\newcommand{\F}{\mathrm{F}}

\newcommand{\eps}{\ensuremath{\varepsilon}}

\newcommand{\ext}{\operatorname{ext}}
\newcommand{\conv}{\operatorname{conv}}
\newcommand{\aconv}{\operatorname{aconv}}

\newcommand{\re}{\operatorname{Re}}
\newcommand{\Id}{\operatorname{Id}}

\newcommand{\Lip}{\operatorname{Lip}}

\newcommand{\spn}{\operatorname{span}}

\newcommand{\smooth}{\operatorname{Smooth}}
\newcommand{\StrExp}{\operatorname{StrExp}}
\newcommand{\Spear}{\operatorname{Spear}}

\renewcommand{\leq}{\leqslant}
\renewcommand{\geq}{\geqslant}

\begin{document}
	\title[Numerical range and BJ-orthogonality]{A numerical range approach to Birkhoff-James orthogonality with applications}
	\date{January 30th, 2024}
	
	\author[M.~Mart\'{\i}n]{Miguel Mart\'{\i}n}
	\address[Mart\'{\i}n]{Universidad de Granada \\ Facultad de Ciencias \\ Departamento de An\'{a}lisis Matem\'{a}tico \\ E-18071 Granada \\ Spain \newline
		\href{http://orcid.org/0000-0003-4502-798X}{ORCID: \texttt{0000-0003-4502-798X} }}
	\email{mmartins@ugr.es}
	\urladdr{\url{https://www.ugr.es/local/mmartins}}
	
	\author[J.~Mer\'{\i}]{Javier Mer\'{\i}}
	\address[Mer\'{\i}]{Universidad de Granada \\ Facultad de Ciencias \\
		Departamento de An\'{a}lisis Matem\'{a}tico \\ E-18071 Granada \\
		Spain \newline
		\href{http://orcid.org/0000-0002-0625-5552}{ORCID: \texttt{0000-0002-0625-5552} }}
	\email{jmeri@ugr.es}
	
	\author[A.~Quero]{Alicia Quero}
	\address[Quero]{Universidad de Granada \\ Facultad de Ciencias \\
		Departamento de An\'{a}lisis Matem\'{a}tico \\ E-18071 Granada \\
		Spain \newline
		\emph{Current address:} Czech Technical University in Prague \\ Faculty of Information Technology \\ Department of Applied Mathematics \\ Th\'{a}kurova 9 \\ 160 00 \\ Prague 6 \\ Czech Republic \newline
		\href{http://orcid.org/0000-0003-4534-8097}{ORCID: \texttt{0000-0003-4534-8097} }}
	\email{aliciaquero@ugr.es}
	
	\author[S.~Roy]{Saikat Roy}
	\address[Roy]{Department of Mathematics\\ Indian Institute of Technology Bombay\\ Mumbai, Maharashtra - 400076\\ India \newline
    \href{https://orcid.org/my-orcid?orcid=0000-0002-7926-1427}{ORCID: \texttt{0000-0002-7926-1427} }}
	\email{saikatroy.cu@gmail.com}
	
	\author[D.~Sain]{Debmalya Sain}
	\address[Sain]{Department of Mathematics \\ Indian Institute of Information Technology Raichur \\ Karnataka-584135 \\ India \newline
    \href{https://orcid.org/0000-0002-9721-1597}{ORCID: \texttt{0000-0002-9721-1597} }}
	\email{saindebmalya1@gmail.com}
	
	\thanks{The first, second, and third named authors have been supported by PID2021-122126NB-C31 funded by MCIN/AEI/ 10.13039/501100011033 and “ERDF A way of making Europe”, by Junta de Andaluc\'ia I+D+i grants P20\_00255 and FQM-185, and by ``Maria de Maeztu'' Excellence Unit IMAG, reference CEX2020-001105-M funded by MCIN/AEI/10.13039/501100011033. The third author is also supported by the Ph.D.\ scholarship FPU18/03057 (MECD). The fourth named author has been supported by the Institute Postdoctoral Fellowship at Indian Institute of Technology Bombay under the supervision of Prof. Sourav Pal. The research of the fifth named author has been sponsored by a Maria Zambrano Fellowship of the University of Granada and “ERDF A way of making Europe”, and grant PID2021-122126NB-C31 funded by MCIN/AEI and by “ERDF A way of making Europe”, under the guidance of the first named author.}
	
	\thispagestyle{plain}
	
	\begin{abstract}
The main aim of this paper is to provide characterizations of Birkhoff-James orthogonality (BJ-orthogonality in short) in a number of families of Banach spaces in terms of the elements of significant subsets of the unit ball of their dual spaces, which makes the characterizations more applicable. The tool to do so is a fine study of the abstract numerical range and its relation with the BJ-orthogonality. Among other results, we provide a characterization of BJ-orthogonality for spaces of vector-valued bounded functions in terms of the domain set and the dual of the target space, which is applied to get results for spaces of vector-valued continuous functions, uniform algebras, Lipschitz maps, injective tensor products, bounded linear operators with respect to the operator norm and to the numerical radius, multilinear maps, and polynomials. Next, we study possible extensions of the well-known Bhatia-\v{S}emrl theorem on BJ-orthogonality of matrices, showing results in spaces of vector-valued continuous functions, compact linear operators on reflexive spaces, and finite Blaschke products. 
Finally, we find applications of our results to the study of spear vectors and spear operators. We show that no smooth point of a Banach space can be BJ-orthogonal to a spear vector of $Z$. As a consequence, if $X$ is a Banach space containing strongly exposed points and $Y$ is a smooth Banach space with dimension at least two, then there are no spear operators from $X$ to $Y$. Particularizing this result to the identity operator, we show that a smooth Banach space containing strongly exposed points has numerical index strictly smaller than one. These latter results partially solve some open problems.
	\end{abstract}
	
	\subjclass{Primary 46B04,  Secondary 46B20, 46B25, 46B28, 46E15, 46E40, 47A12, 47A30}
	\keywords{Birkhoff-James orthogonality; numerical range, Bhatia-\v{S}emrl results; smoothness, bounded linear operators; spear vectors and operators; numerical index}
	
	\dedicatory{To Professor William B.~Johnson on the occasion of his 80th birthday}
	
	\maketitle
	
	\section{Introduction}
	Let $Z$ be a Banach space over the field $\K$ (which will always be considered as $\R$ or $\C$). Given $x,y\in Z$, we say that $x$ is \emph{Birkhoff-James orthogonal} to $y$ (\emph{BJ-orthogonal} in short), denoted by $x \perp_B y$, if
	$$\|x + \lambda y\|\geq \|x\| \quad \forall \lambda\in\K.$$
	This definition, proposed by Birkhoff \cite{Birkhoff} in the setting of metric linear spaces, has a natural geometric interpretation: $x \perp_B y$ if and only if the (real or complex) line $\{x+\lambda y\colon \lambda\in\K\}$ is disjoint from the open ball of radius $\|x\|$ centered at the origin. Observe that BJ-orthogonality is homogeneous, i.e., $x\perp_B y$ implies that $\alpha x \perp_B \beta y$ for every $\alpha,\beta\in\K$. Also, smoothness of the norm of $Z$ at $x$ is equivalent to the right-additivity of $\perp_B$ at $x$: $x$ is smooth in $Z$ if and only if for any $y,z\in Z$, $ x \perp_B y$, $x \perp_B z $ together imply that $x \perp_B (y+z)$. In case the norm is induced by an inner product $ \langle~,~ \rangle $, it is elementary to notice that BJ-orthogonality is equivalent to the usual orthogonality: $x \perp y$ if and only if $\langle x,y\rangle = 0 $ if and only if $ x \perp_B y. $ This shows that BJ-orthogonality generalizes the concept of usual orthogonality to the framework of norms. It is worth mentioning that BJ-orthogonality is not a symmetric relation in general, i.e., $x\perp_B y$ may not necessarily imply $y\perp_B x$. Although there exist several non-equivalent notions of orthogonality in Banach spaces, it is commonly accepted that BJ-orthogonality is arguably the most useful one amongst them by virtue of its rich connections with many important concepts in the geometric theory of Banach spaces, including smoothness, operator norm attainment, characterizations of Euclidean and Hilbert spaces among Banach spaces, and best approximations. We refer the interested readers to \cite{Sain-JMAA-2018, Sain-PAMS, SainPaul-2013, SainPaulMalRay, Singer, Stampfli}, and the references therein, for more information in this regard.

	A general way to study BJ-orthogonality  in any Banach space $Z$ was given by R.~C.~James in terms of the dual space $Z^*$ of $Z$.
	
	\begin{fact}[{\cite[Corollary~2.2]{James}}]\label{fact:James}
		Let $Z$ be a Banach space and let $x,y \in Z$. Then,
		$$
		x\perp_By \ \Longleftrightarrow\ \text{there exists $\phi\in Z^*$ with $\|\phi\|=1$ such that $\phi(x)=\|x\|$ and $\phi(y)=0$.}
		$$
	\end{fact}
	
This characterization of BJ-orthogonality immediately relates it with the norm attainment problem for functionals in the dual space. As a matter of fact, one of the useful ways to reap the benefits out of the concept of BJ-orthogonality for the purpose of understanding the geometric and analytic structures of a Banach space, is to apply James' characterization of BJ-orthogonality in the corresponding dual space. As we will see in this article, it is possible to obtain further refinements of the James characterization in many important cases including the Banach space of bounded linear operators between Banach spaces.

	The above result by James also relates BJ-orthogonality with the concept of (abstract) numerical range. Let us introduce the required notations and definitions. Given a Banach space $Z$, we write $B_Z$ and $S_Z$ to denote, respectively, the closed unit ball and the unit sphere of $Z$, $\re(\cdot)$ will denote the real part (which is nothing but the identity if we are dealing with real numbers), and we write $\T$ for the set of modulus-one scalars. If $u\in Z$ is a norm-one element, the \emph{(abstract) numerical range} of $z\in Z$ with respect to $(Z,u)$ is the non-empty compact convex subset of $\K$ given by
	$$
	V(Z,u,z):=\{\phi(z) \colon \phi\in \F(B_{Z^*},u)\},
	$$
	where $\F(B_{Z^*},u):=\{\phi\in S_{Z^*}\colon \phi(u)=1\}$ is the \emph{face} of $B_{Z^*}$ generated by $u$, also known as the \emph{set of states of $Z$ relative to $u$}. The concept of abstract numerical range takes its roots in a 1955 paper by Bohnenblust and Karlin \cite{Bohn-Karlin} and it was introduced in the 1985 paper \cite{MarMenaPayaRod1985}. We refer the reader to the classical books \cite{BonsallDuncan1,B-D2} by Bonsall and Duncan, to Sections 2.1 and 2.9 of the book \cite{Cabrera-Rodriguez}, and to Section 2 of \cite{KMMPQ} for more information and background.
	
	Observe that, with the definition of numerical range in hands, Fact~\ref{fact:James} can be easily written in the following way.
	
	\begin{proposition}\label{prop:BJ-using-numranges}
		Let $Z$ be a Banach space, let $u\in S_Z$, and let $z\in Z$. Then,
		$$
		u\perp_Bz\ \Longleftrightarrow\ 0\in V(Z,u,z).
		$$
	\end{proposition}
	
	Let us also comment that it is also possible to write the numerical range in terms of the BJ-orthogonality, see Proposition~\ref{prop:numrange-using-BJ}. It is then clear that the study of BJ-orthogonality and the study of abstract numerical ranges are somehow equivalent.
	
	The main disadvantage of Proposition~\ref{prop:BJ-using-numranges} (and of Fact~\ref{fact:James}) is that we have to deal with the whole dual of the Banach space $Z$, and this is difficult in many situations. For instance, when $Z$ is a space of bounded linear operators, which is the most interesting case for us, the dual space is a wild object that is not easy to manage. For an easier writing of our discussion, let us introduce the following notation: given Banach spaces $X$, $Y$, we write $\mathcal{L}(X,Y)$ to denote the space of all bounded linear operators from $X$ to $Y$ and $\mathcal{K}(X,Y)$ for its subspace consisting of compact operators. When $X=Y$, we just write $\mathcal{L}(X)$ and $\mathcal{K}(X)$. In the case when $Z$ is the space of $n\times n$ matrices (identified with $\mathcal{L}(H)$ where $H$ is an $n$-dimensional Hilbert space), a celebrated result by Bhatia and \v{S}emrl \cite[Theorem~1.1]{BhatiaSmerl} says that two matrices $A$, $B$ satisfy that $A\perp_B B$ if and only if there is a norm-one vector $x$ such that $\|Ax\|=\|A\|$ and $\langle Ax,Bx\rangle =0$ (that is, there is a norm-one vector $x$ at which $A$ attains its norm and such that $Ax\perp Bx$). Observe that it is equivalent to say that, in this case, when $A\perp_B B$, the functional $\phi$ on the space of $n\times n$ matrices given by Fact~\ref{fact:James} can be taken of the form $\phi(C)=\langle Cx,y\rangle$ for some norm-one vectors $x$ and $y$ (see the proof of Corollary~\ref{corollary:BhatiaSmerltheorem}). Clearly, this gives much more information than the one provided by Fact~\ref{fact:James} and avoids to deal with the dual of the space of matrices. This result does not extend to general operators on infinite-dimensional Hilbert spaces (as they do not need to attain the norm), but there is a similar result: given two bounded linear operators $A$ and $B$ on a Hilbert space $H$, $A\perp_B B$ if and only if there is a sequence $\{x_n\}$ in $S_H$ satisfying that $\lim \|Ax_n\|=\|A\|$ and $\lim \langle Ax_n,Bx_n\rangle =0$ \cite[Lemma~2.2]{Magajna}, \cite[Remark~3.1]{BhatiaSmerl}. The significance of the result obtained by Bhatia and \v{S}emrl lies in the fact that it allows us to examine the orthogonality of bounded linear operators on a Hilbert space in terms of the usual orthogonality of certain special elements in the ground space. We would like to emphasize here that such a characterization of BJ-orthogonality is certainly more handy than James' characterization, since we do not need to deal with the dual of the operator space. Moreover, as already mentioned in \cite{BhatiaSmerl}, it is natural to speculate about the validity of the above results in the case of bounded linear operators on a Banach space. In general, they do not extend to operators between general Banach spaces, even in the finite-dimensional case, as it was shown by Li and Schneider \cite[Example~4.3]{LiSchneider}. However, a related weaker result was proved in the same paper \cite[Proposition~4.2]{LiSchneider}: if $X$ and $Y$ are finite-dimensional Banach spaces and $T,A\in \mathcal{L}(X,Y)$, then
	$$
	T\perp_B A \ \Longleftrightarrow \
	0\in \conv\left(\bigl\{y^*(Ax)\colon x\in \ext(B_X), \ y^*\in \ext(B_{Y^*}), \ y^*(Tx)=\|T\|\bigr\}\right),
	$$
	where $\ext(C)$ denotes the set of extreme points of a convex set $C$ and $\conv(\cdot)$ is the convex hull. Observe that this result is similar to Bhatia-\v{S}emrl's one up to taking the convex hull in $\K$.  How did Li and Schneider get this result? Just by characterizing the extreme points of the dual unit ball of $\mathcal{L}(X,Y)$ when $X$ and $Y$ are finite-dimensional and then using a classical result by Singer about best approximation. Let $Z$ be a Banach space, let $M$ be a subspace of $Z$, and let $x\in Z$. An element $m_0\in M$ is said to be a \emph{best approximation} of $x$ at $M$ if
	$$
	\|x-m_0\|\leq\|x-m\| \quad \forall m\in M.
	$$
	Observe that $m_0$ is a best approximation to $x$ in $M$ if and only if $x-m_0$ is BJ-orthogonal to $M$, i.e., $x-m_0\perp_B m$ for every $m\in M$. Equivalently, given $x,y\in Z$, $x\perp_B y$ if and only if $0$ is a best approximation to $x$ in $\spn\{y\}$. We refer the interested reader to the classical book \cite{Singer} by I.~Singer for background.  Using the relation between best approximation and BJ-orthogonality, the classical result of I.~Singer that Li and Schneider used reads as follows.
	
	\begin{fact}[{\cite[Theorem~II.1.1]{Singer}}]\label{fact:Singer}
		Let $Z$ be a Banach space and let $u,z \in Z$. Then,
		$$
		u \perp_B z \ \Longleftrightarrow\ 0\in \conv\left(\bigl\{\phi(z)\colon \phi \in\ext(B_{Z^*}),\ \phi(u)=\|u\| \bigr\} \right).
		$$
	\end{fact}
	
	Observe that this result just says that the functional $\phi$ in Fact~\ref{fact:James} can be taken in the convex hull of the set of extreme points of $B_{Z^*}$. We will provide in Proposition~\ref{prop:num-range-ext} a version of Proposition~\ref{prop:BJ-using-numranges} using only extreme points with an independent proof. Of course, Fact~\ref{fact:Singer} is very interesting in the cases when the extreme points of the dual ball are known and easy to manage: for $Z=C(K)$ or for $Z$ being an isometric predual of an $L_1(\mu)$ space, or even when $Z=\mathcal{L}(X,Y)$ and $X$ and $Y$ are finite-dimensional (as it was done by Li and Schneider, see \cite[Proposition~4.2]{LiSchneider}). Actually, for arbitrary spaces $X$ and $Y$, the  extreme points of the dual ball of $Z=\mathcal{K}(X,Y)$ have been described as $\ext(B_{X^{**}})\otimes \ext(B_{Y^*})$ (see \cite[Theorem 1.3]{Ruess-Stegall-MathAnn1982} for the real case and \cite[Theorem 1]{LimaOlsen} for the complex case). In the particular case when $X$ is reflexive, this provides the following result which covers Li and Schneider's one: let $X$ be a reflexive space, let $Y$ be a Banach space, and let $T,A\in \mathcal{K}(X,Y)$; then
	$$
	T\perp_B A
	\ \Longleftrightarrow \
	0\in \conv\left(\bigl\{y^*(Ax)\colon x\in \ext(B_X), \ y^*\in \ext(B_{Y^*}), \ y^*(Tx)=\|T\|\bigr\}\right)
	$$
	(see Corollary~\ref{B-J-orth-operator-norm-compact}). When we deal with non-compact operators, there is no description of the extreme points of the unit ball of $\mathcal{L}(X,Y)^*$ available, hence Fact~\ref{fact:Singer} is not applicable in this case. However, a somehow similar result was proved in \cite[Theorem~2.2]{Roypre2022}: let $X$, $Y$ be Banach spaces and let $T,A\in\mathcal{L}(X,Y)$; then,
	\begin{equation*}
		T\perp_B A \Longleftrightarrow 0\in \conv\left(\bigl\{\lim y_n^*(Ax_n)\colon (x_n,y_n^*)\in S_X \times S_{Y^*} \ \forall n\in\N, \ \lim y_n^*(Tx_n)=\|T\|\bigr\}\right).
	\end{equation*}
	This is, as far as we know, the most general result concerning a characterization of BJ-orthogonality of operators in terms of the domain and range spaces and their duals.
	
	Our main aim in this paper is to provide a very general result characterizing BJ-orthogonality in a Banach space $Z$ in terms of the actions of elements on an arbitrary one-norming subset. Recall that a subset $\Lambda\subset S_{Z^*}$ is said to be \emph{one-norming} for $Z$ if $\|z\|=\sup\{|\phi(z)|\colon \phi\in \Lambda\}$ for all $z\in Z$ (equivalently, if $B_{Z^*}$ equals the absolutely weak-star closed convex hull of $\Lambda$). One of the assertions of this general result (see Corollary~\ref{cor:B-J-orth-Z-Lambda}) is the following: let $Z$ be a Banach space, $\Lambda\subset S_{Z^*}$ be one-norming for $Z$; then for $u\in S_Z$ and $z\in Z$,
	\begin{equation*}
		u\perp_B z \ \Longleftrightarrow \ 0\in \conv\left(\bigl\{\lim \psi_n(z)\overline{\psi_n(u)}\colon \psi_n\in \Lambda, \lim |\psi_n(u)|=1\bigr\}\right).
	\end{equation*}
	The way to get the result is to combine Proposition~\ref{prop:BJ-using-numranges} with a very general result on numerical ranges, Theorem~\ref{thm:num-range-Lambda}, which extends previous characterizations from \cite{KMMPQ}. This result also allows to characterize smooth points, see Corollary~\ref{cor:characterization-smooth-lambda}. There are also nicer versions of these results in the case when instead of a one-norming subset $\Lambda$, we have a subset $C$ of $S_{Z^*}$ such that its weak-star closed convex hull is the whole $B_{Z^*}$, see Theorem~\ref{theorem:num-range-C} and Corollaries \ref{cor:B-J-orth-Z-C} and \ref{cor:characterization-smooth}. All of this is the content of  Section~\ref{sect:B-J-orth-smooth-num-range} of this manuscript.
	
	Section \ref{section:someparticularcases} contains a number of particular cases in which the results of Section~\ref{sect:B-J-orth-smooth-num-range} apply. It is divided into several subsections, and covers results in a number of spaces. Even though some of the results of this section were previously known, the previous approaches were different and use ad hoc techniques for each of the particular cases, while our present approach generalizes all these techniques. On the other hand, the general result for $\ell_\infty(\Gamma,Y)$ we give in Theorem~\ref{theorem:ellinftygammaY} seems to be new, as are its applications for spaces of vector-valued continuous functions (Corollaries \ref{corollary (to be used)} and \ref{corollary:C(K,Y)extreme}), uniform algebras (Corollary~\ref{corollary:uniformalgebras}), Lipschitz maps (Proposition~\ref{prop:Lipschitz}), and injective tensor products (Proposition~\ref{prop:B-J-orth-injective-tensor}). For bounded linear operators (Subsection~\ref{subsection:operators}), most of the results have already been known, but there are some improvements of previous results in Proposition~\ref{prop:compact-operators-general-extreme} and Corollary~\ref{B-J-orth-operator-norm-compact}. Besides, we include a result on smoothness of bounded linear operators which will be used in Section~\ref{section:An-application}. Subsection~\ref{subsection:multilinear-polynomials} deals with multilinear maps and polynomials and the results seem to be new. Finally, Subsection~\ref{subsection:numericalradiusnorm} contains results on BJ-orthogonality with respect to the numerical radius of operators which were previously known.
	
	Next, in Section~\ref{section:BS-kindofresults} we provide several results related to the Bhatia-\v{S}emrl theorem (in the sense of removing the convex hull and the limits of the characterization of BJ-orthogonality). The main result (Theorem~\ref{Compact continuous}) is about vector-valued continuous functions on a compact Hausdorff space and seems to be completely new. As consequences, we obtain Bhatia-\v{S}emrl's kind of results for compact operators on reflexive Banach spaces, Proposition~\ref{prop:BS-realcase} for the real case, Theorem~\ref{theorem:newBS-result-complex} in the complex case, and the latter is new for infinite-dimensional spaces. We also obtain a nice characterization of BJ-orthogonality for finite Blaschke products (Example~\ref{example:Blaschke}).
	
Finally, Section~\ref{section:An-application} contains applications of the results in the paper to the study of spear vectors, spear operators, and Banach spaces with numerical index one. The notions of spear vectors and spear operators will be defined in Section~\ref{section:An-application}. The results in this section are actually consequences of Theorem~\ref{theorem:smooth-not-orthogonal} which says that if $u$ is a vertex of a Banach space $Z$ and $z\in Z$ is smooth in $(Z,v_u)$, then $z$ cannot be BJ-orthogonal to $u$ in $(Z,v_u)$. Thus, no smooth point of a Banach space $Z$ can be BJ-orthogonal to a spear vector of $Z$ (Corollary~\ref{cor:smooth-orth-spear}). The particularization of the results to the case $Z=\mathcal{L}(X,Y)$ leads to obstructive results for the existence of spear operators. In particular, we show that if $X$ is a Banach space containing strongly exposed points and $Y$ is a smooth Banach space with dimension at least two, then there are no spear operators in $\mathcal{L}(X,Y)$ (Corollary~\ref{corollary:strexpnonempty-smooth-smooth-ortogona}) and this result is proved using a sufficient condition for an operator to be smooth  (Proposition~\ref{prop:suff-condition-smooth-operator}). This result somehow extends \cite[Proposition~6.5.a]{KMMP-SpearsBook} and provides a partial answer to \cite[Problem~9.12]{KMMP-SpearsBook}. Particularizing this to the identity operator, we get an obstructive condition for a Banach space to have numerical index one: the existence of a smooth point which is BJ-orthogonal to a strongly exposed point (Corollary~\ref{cor:X-exp-point-n(X)<1}). In particular, smooth Banach spaces with dimension at least two containing strongly exposed points do not have numerical index one (Corollary~\ref{corollary:smooth-stronglyexposednotempty-nonX=1}). This latter result is a partial answer to the question of whether a smooth Banach space of dimension at least two may have numerical index one \cite{ConvexSmooth}. Let us comment that the mix of ideas from numerical ranges and from BJ-orthogonality is the key to obtaining these interesting applications which partially solve some open questions. Moreover, the abstract numerical range approach to BJ-orthogonality considered in this article generalizes all of the previously mentioned characterizations to a much broader framework. In view of this, it is reasonable to expect that the methods developed here will cover more particular cases, known and new.
	
	A preprint version of this manuscript (which is very close to the present one) is included as Chapter VI in Alicia Quero's PhD dissertation \cite{PhD-Alicia}, which followed a compendium form and was defended at the University of Granada in September 2023.
	
	\section{The numerical range approach}\label{sect:B-J-orth-smooth-num-range}
	The aim of this section is to connect BJ-orthogonality and smoothness with the theory of abstract numerical ranges and present different expressions of the abstract numerical range which will be very useful in order to characterize BJ-orthogonality  and smoothness in several contexts.
	
	Let us start with a result showing that the abstract numerical range can be expressed in terms of BJ-orthogonality. This result, together with Proposition~\ref{prop:BJ-using-numranges}, shows that the study of BJ-orthogonality and the study of abstract numerical ranges are somehow equivalent.
	
	\begin{proposition}\label{prop:numrange-using-BJ}
		Let $Z$ be a Banach space and let $u\in S_Z$. Then, for every $z\in Z$,
		$$
		V(Z,u,z)=\left\{\alpha\in\K\colon u\perp_B(z-\alpha u)\right\}.
		$$
	\end{proposition}
	
	\begin{proof}
		Let $\alpha\in V(Z,u,z)$, then there exists $\phi\in S_{Z^*}$ such that $\phi(u)=1$ and $\phi(z)=\alpha$. Thus $\phi(z-\alpha u)=0$ and so $u\perp_B(z-\alpha u)$. Conversely, if $\alpha\in \K$ is such that $u\perp_B(z-\alpha u)$, then there exists $\phi\in S_{Z^*}$ such that $\phi(u)=1$ and $\phi(z-\alpha u)=\phi(z)-\alpha=0$, therefore $\alpha=\phi(z)\in V(Z,u,z)$.
	\end{proof}
	
	Let $Z$ be a Banach space and let $u\in S_Z$. Our aim here is to show how to describe the abstract numerical range $V(Z,u,\cdot)$ in terms of a fixed one-norming subset $\Lambda\subset S_{Z^*}$ which will allow us to get  characterizations of BJ-orthogonality and smoothness. In the particular case in which $\Lambda$ is equal to $\ext(B_{Z^*})$, the characterization of BJ-orthogonality actually follows from Fact~\ref{fact:Singer}. But we are also able to get a result on abstract numerical ranges as an easy consequence of the Bauer Maximum Principle. Observe that Fact~\ref{fact:Singer} can also be deduced from the next proposition and Proposition~\ref{prop:BJ-using-numranges}.
	
	\begin{proposition}\label{prop:num-range-ext}
		Let $Z$ be a Banach space and let $u\in S_Z$. Then, for every $z\in Z$,
		$$
		V(Z,u,z)=\conv\left\{\phi(z)\colon \phi\in\ext\left(B_{Z^*}\right), \, \phi(u)=1\right\}.
		$$
	\end{proposition}
	
	\begin{proof}
		We apply the Bauer Maximum Principle (see \cite[7.69]{AliprantisBorder}, for instance) to the set $\F(B_{Z^*},u)$, which is convex and $w^*$-compact, and to the function $\phi\longmapsto\re\phi(z)$ from $\F(B_{Z^*},u)$ to $\R$, which is $w^*$-continuous and convex. Then, this function attains its maximum at an extreme point of $\F(B_{Z^*},u)$ (which is also an extreme point of $B_{Z^*}$ since $\F(B_{Z^*},u)$ is an extremal subset). That is,
		$$
		\max \re V(Z,u,z)=\max \re \left\{\phi(z)\colon \phi\in\ext\left(B_{Z^*}\right), \ \phi(u)=1\right\}.
		$$
		Now, the result follows using that $V(Z,u,\theta z)=\theta V(Z,u,z)$ and
		$$
		\left\{\phi(\theta z)\colon \phi\in\ext\left(B_{Z^*}\right), \, \phi(u)=1\right\}=
		\theta  \left\{\phi(z)\colon \phi\in\ext\left(B_{Z^*}\right), \, \phi(u)=1\right\}
		$$
		for every $\theta\in\T$.
	\end{proof}
	
	There are Banach spaces $Z$ for which the set of extreme points of the dual space is not known (for instance, this is the case for $Z=\mathcal{L}(X,Y)$ in general). In those cases, another way to characterize the numerical range is needed. This was done in \cite[Propostion 2.14]{KMMPQ} substituting the set of extreme points of the dual ball by a subset $C\subseteq B_{Z^*}$ such that $B_{Z^*}=\overline{\conv}^{w^*}(C)$. We next give a reformulation of that result which will be useful in applications.
	
	\begin{theorem}\label{theorem:num-range-C}
		Let $Z$ be a Banach space, let $u\in S_Z$, and let $C\subseteq B_{Z^*}$ be such that $B_{Z^*}=\overline{\conv}^{w^*}(C)$. Then,
		\begin{align*}
			V(Z,u,z)
			&=\conv\,\bigcap\nolimits_{\delta>0} \overline{\bigl\{\phi(z)\colon \phi\in C,\, \re \phi(u)>1-\delta\bigr\}} \\
			&=\conv \left(\bigl\{\lim \phi_n(z)\colon \phi_n\in C \ \forall n \in \N, \ \lim \phi_n(u)=1 \bigr\}\right)
		\end{align*}
		for every $z\in Z$.
	\end{theorem}
	
	Let us comment that comparing this theorem with Proposition~\ref{prop:num-range-ext}, we lose information as we have to deal with limits, but we obtain a lot of generality, as there are many situations in which $B_{Z^*}=\overline{\conv}^{w^*}(C)$ holds and $C$ is completely different from $\ext(B_{Z^*})$ (even disjoint). In what follows, and in the rest of the paper, when we write $\lim z_n$ for a \emph{bounded scalar sequence} $\{z_n\}$ we are understanding that the sequence is convergent.
	
	\begin{proof}[Proof of Theorem~\ref{theorem:num-range-C}]
		The first equality was already proved in \cite[Propostion 2.14]{KMMPQ}, let us prove that
		\begin{align*}
			V(Z,u,z)
			&=\conv \left(\bigl\{\lim \phi_n(z)\colon \phi_n\in C \ \forall n \in \N, \ \lim \phi_n(u)=1 \bigr\}\right).
		\end{align*}
		For $z\in Z$, we write $W(z):=\{\lim \phi_n(z)\colon \phi_n\in C \ \forall n \in \N, \ \lim \phi_n(u)=1 \bigr\}$ and we prove first the inclusion $V(Z,u,z)\supseteq \conv W(z)$. Given $\lambda_0\in W(z)$, for each $n\in\N$ there exists $\phi_n\in C$ such that
		$$
		\left|\phi_n(z)-\lambda_0\right|<1/n \quad \textnormal{and} \quad \left|\phi_n(u)-1\right|<1/n.
		$$
		Since $B_{Z^*}$ is $w^*$-compact, there is a $w^*$-cluster point $\phi_0\in B_{Z^*}$ of the sequence $\{\phi_n\}_{n\in\N}$. Then, it follows that $\phi_0(z)=\lambda_0$ and $\phi_0(u)=1$, so $\lambda_0\in V(Z,u,z)$ and the desired inclusion holds by the convexity of $V(Z,u,z)$.
		
		To prove the reverse inclusion, it is enough to show that the inequality
		\begin{equation}\label{eq:lemma-num-range-C}
			\sup \re V(Z,u,z)\leq \sup \re W(z)
		\end{equation}
		holds for every $z\in Z$, as $V(Z,u,\theta z)=\theta V(Z,u,z)$ and $W(\theta z)=\theta W(z)$ for every $\theta\in\T$, and $W(z)$ is closed. So, for fixed $z\in Z$ and $\phi_0\in \F(B_{Z^*},u)$, we apply \cite[Lemma 2.15]{KMMPQ} for $\delta=1/n$ to obtain a sequence $\{\phi_n\}_{n\in\N}$ in $C$ such that
		$$
		\re\phi_n(u)>1-1/n \quad \textnormal{and} \quad \re\phi_n(z)>\re\phi_0(z)-1/n.
		$$
		We may and do suppose (up to taking a subsequence, if needed), that the sequences $\{\phi_n(u)\}$ and $\{\phi_n(z)\}$ are convergent. Therefore, we get $\lim \phi_n(u)=1$ and $\re \phi_0(z)\leq \re\lim \phi_n(z)\leq \sup \re W(z)$, and so inequality \eqref{eq:lemma-num-range-C} follows.
	\end{proof}
	
	We are now able to generalize the previous result to the case of one-norming subsets. We will include more characterizations here as this is the most general result that we have.
	
	\begin{theorem}\label{thm:num-range-Lambda}
		Let $Z$ be a Banach space, let $u\in S_Z$, and let $\Lambda\subset B_{Z^*}$ be one-norming for $Z$. Then,
		\begin{align*}
			V(Z,u,z)&=\conv\left(\bigl\{\theta_0 \lim\psi_n(z)\colon \psi_n\in \Lambda \ \forall n\in \N,\ \lim\psi_n(u)=\overline{\theta_0},\ \theta_0\in \T\bigr\}\right) \\
			&=\conv\left(\left\{\lim\psi_n(z)\overline{\psi_n(u)}\colon \psi_n\in \Lambda \ \forall n\in \N,\ \lim|\psi_n(u)|=1\right\}\right) \\
			&=\conv\,\bigcap\nolimits_{\delta>0} \overline{\bigl\{\psi(z)\overline{\psi(u)} \colon \psi\in\Lambda, \ |\psi(u)|>1-\delta\bigr\}} \\
			&=\bigcap\nolimits_{\delta>0} \conv \overline{\bigl\{\psi(z)\overline{\psi(u)} \colon \psi\in\Lambda, \ |\psi(u)|>1-\delta\bigr\}}
		\end{align*}
		for every $z\in Z$.
	\end{theorem}
	
	\begin{proof}
		We start proving the first three equalities. To do so, we apply Theorem~\ref{theorem:num-range-C} for $C=\T \Lambda$ which satisfies $\overline{\conv}^{w^*}(C)=\overline{\aconv}^{w^*}(\Lambda)=B_{Z^*}$ to obtain that
		$$
		V(Z,u,z)=\conv\left(\bigl\{\lim \phi_n(z)\colon \phi_n\in \T\Lambda \ \forall n \in \N, \ \lim \phi_n(u)=1 \bigr\}\right).
		$$
		So it is enough to show the following chain of inclusions:
		\begin{align*}
			\bigl\{\lim \phi_n(z)\colon \phi_n\in \T\Lambda \ \forall n \in \N, \ \lim \phi_n(u)=1 \bigr\}&\subseteq \bigl\{\theta_0 \lim\psi_n(z)\colon \psi_n\in \Lambda \ \forall n\in \N,\ \lim\psi_n(u)=\overline{\theta_0}\bigr\} \\
			&\subseteq \left\{\lim\psi_n(z)\overline{\psi_n(u)}\colon \psi_n\in \Lambda \ \forall n\in \N,\ \lim|\psi_n(u)|=1\right\} \\
			&\subseteq \bigcap\nolimits_{\delta>0} \overline{\bigl\{\psi(z)\overline{\psi(u)} \colon \psi\in\Lambda, \ |\psi(u)|>1-\delta\bigr\}} \\
			&\subseteq \bigl\{\lim \phi_n(z)\colon \phi_n\in \T\Lambda \ \forall n \in \N, \ \lim \phi_n(u)=1 \bigr\}.
		\end{align*}
		For the first inclusion, take $\{\phi_n\}_{n\in\N} \subseteq \T\Lambda$ with $\lim \phi_n(u)=1$, then $\phi_n=\theta_n\psi_n$ for $\theta_n\in\T$ and $\psi_n\in\Lambda$. Let $\theta_0\in\T$ be a cluster point of $\{\theta_n\}_{n\in\N}$, then $$\lim\psi_n(u)=\lim\overline{\theta_n}\phi_n(u)=\overline{\theta_0} \quad \textnormal{and} \quad \theta_0\lim\psi_n(z)=\theta_0\lim\overline{\theta_n}\phi_n(z)=\lim\phi_n(z).$$
		The second inclusion is evident. For the third one, let $\lambda=\lim\psi_n(z)\overline{\psi_n(u)}$ for some sequence $\{\psi_n\}_{n\in\N}\subseteq \Lambda $ with $\lim |\psi_n(u)|=1$ and fix $\delta>0$. There exists $n_0\in\N$ such that
		$$
		\left|\lambda-\psi_n(z)\overline{\psi_n(u)}\right|<1/n \quad \textnormal{and} \quad |\psi_n(u)|>1-1/n>1-\delta
		$$
		for every $n\geq n_0$, therefore $\lambda\in\overline{\bigl\{\psi(z)\overline{\psi(u)} \colon \psi\in\Lambda, \ |\psi(u)|>1-\delta\bigr\}}$ and the arbitrariness of $\delta$ gives the inclusion.
		\\
		For the last inclusion, let $\lambda\in\bigcap\nolimits_{\delta>0} \overline{\bigl\{\psi(z)\overline{\psi(u)} \colon \psi\in\Lambda, \ |\psi(u)|>1-\delta\bigr\}}$. For every $n\in\N$ there exists $\psi_n\in\Lambda$ such that
		$$
		\left|\lambda-\psi_n(z)\overline{\psi_n(u)}\right|<1/n \quad \textnormal{and} \quad |\psi_n(u)|>1-1/n. $$
		For each $n\in\N$, take $\theta_n\in\T$ such that $\theta_n\psi_n(u)=|\psi_n(u)|$ and define $\phi_n=\theta_n\psi_n\in\T\Lambda$. Then, we have that
		$$
		\lim\phi_n(u)=\lim\theta_n\psi_n(u)=\lim |\psi_n(u)|=1
		$$
		and
		$$
		\lambda=\lim\psi_n(z)\overline{\psi_n(u)}=\lim \theta_n\psi_n(z)\overline{\theta_n\psi_n(u)}=\lim\phi_n(z)|\psi_n(u)|=\lim\phi_n(z),
		$$
		which finishes the proof of the chain of inclusions.
		
		Finally, in order to prove the last equality of the theorem, observe that
		$$
		\conv\,\bigcap\nolimits_{\delta>0} \overline{\bigl\{\psi(z)\overline{\psi(u)} \colon \psi\in\Lambda, \ |\psi(u)|>1-\delta\bigr\}} \subseteq \bigcap\nolimits_{\delta>0} \conv \overline{\bigl\{\psi(z)\overline{\psi(u)} \colon \psi\in\Lambda, \ |\psi(u)|>1-\delta\bigr\}}
		$$
		and let us show that the latter set is contained in $V(Z,u,z)$. \\
		For fixed $\lambda\in\bigcap\nolimits_{\delta>0} \conv \overline{\bigl\{\psi(z)\overline{\psi(u)} \colon \psi\in\Lambda, \ |\psi(u)|>1-\delta\bigr\}}$, it is clear that $$\lambda \in \conv\overline{\bigl\{\psi(z)\overline{\psi(u)} \colon \psi\in\Lambda, \ |\psi(u)|>1-1/n\bigr\}}$$ for every $n\in\N$, and we apply Carath\'{e}odory's Theorem to obtain the existence of $a_n, b_n, c_n\in[0,1]$ with $a_n+b_n+c_n=1$ and $\phi_n, \psi_n, \xi_n\in \Lambda$ such that
		\begin{align}
			&|\phi_n(u)|\geq1-1/n, \quad |\psi_n(u)|\geq1-1/n, \quad |\xi_n(u)|\geq1-1/n, \quad \textnormal{and} \label{eq1:proof-lemma-num-range-Lambda}
			\\
			&\left|\lambda-\bigl(a_n\phi_n(z)\overline{\phi_n(u)}+b_n\psi_n(z)\overline{\psi_n(u)}+c_n\xi_n(z)\overline{\xi_n(u)}\bigr)\right|<1/n \label{eq2:proof-lemma-num-range-Lambda}
		\end{align} for every $n\in \N$. We may find $a,b,c\in[0,1]$ and subsequences $\{a_{\sigma(n)}\}_{n\in\N}$, $\{b_{\sigma(n)}\}_{n\in\N}$, $\{c_{\sigma(n)}\}_{n\in\N}$ of $\{a_n\}_{n\in\N}$, $\{b_n\}_{n\in\N}$, $\{c_n\}_{n\in\N}$ respectively such that $\{a_{\sigma(n)}\}_{n\in\N}\to a$, $\{b_{\sigma(n)}\}_{n\in\N}\to b$, $\{c_{\sigma(n)}\}_{n\in\N}\to c$ and $a+b+c=1$. Additionally, by passing to a subsequence, we may assume that $\{\phi_{\sigma(n)}(u)\}_{n\in\N}$, $\{\psi_{\sigma(n)}(u)\}_{n\in\N}$, $\{\xi_{\sigma(n)}(u)\}_{n\in\N}$, $\{\phi_{\sigma(n)}(z)\}_{n\in\N}$, $\{\psi_{\sigma(n)}(z)\}_{n\in\N}$, and $\{\xi_{\sigma(n)}(z)\}_{n\in\N}$ are convergent. Since $B_{Z^*}$ is $w^*$-compact, let $\phi_0, \psi_0, \xi_0 \in B_{Z^*}$ be $w^*$-cluster points of the sequences $\{\phi_{\sigma(n)}\}_{n\in\N}$, $\{\psi_{\sigma(n)}\}_{n\in\N}$, $\{\xi_{\sigma(n)}\}_{n\in\N}$ respectively. Then,
		\begin{align*}
			&\lim\phi_{\sigma(n)}(u)=\phi_0(u),  &&\lim\psi_{\sigma(n)}(u)=\psi_0(u),  &&\lim\xi_{\sigma(n)}(u)=\xi_0(u), \\
			&\lim\phi_{\sigma(n)}(z)=\phi_0(z), &&\lim\psi_{\sigma(n)}(z)=\psi_0(z), &&\lim\xi_{\sigma(n)}(z)=\xi_0(z).
		\end{align*}
		It follows from \eqref{eq1:proof-lemma-num-range-Lambda} that $\left|\phi_0(u)\right|=\left|\psi_0(u)\right|=\left|\xi_0(u)\right|=1$. Define
		$$
		\Phi=a\overline{\phi_0(u)}\phi_0+b\overline{\psi_0(u)}\psi_0+c\overline{\xi_0(u)}\xi_0\in B_{Z^*},
		$$
		then $\Phi(u)=a\left|\phi_0(u)\right|^2+b\left|\psi_0(u)\right|^2+c\left|\xi_0(u)\right|^2=1$ and
		$$\lambda=\lim \left(a_{\sigma(n)}\phi_{\sigma(n)}(z)\overline{\phi_{\sigma(n)}(u)}+b_{\sigma(n)}\psi_{\sigma(n)}(z)\overline{\psi_{\sigma(n)}(u)}+c_{\sigma(n)}\xi_{\sigma(n)}(z)\overline{\xi_{\sigma(n)}(u)}\right)=\Phi(z)$$
		by \eqref{eq2:proof-lemma-num-range-Lambda}, which imply that $\lambda\in V(Z,u,z)$.
	\end{proof}
	
	Thanks to the different expressions of the numerical range provided in Theorems \ref{theorem:num-range-C} and \ref{thm:num-range-Lambda}, we are able to give new characterizations of BJ-orthogonality.
	
	\begin{corollary}\label{cor:B-J-orth-Z-C}
		Let $Z$ be a Banach space, let $u\in S_Z$, and let $C\subseteq B_{Z^*}$ be such that $B_{Z^*}=\overline{\conv}^{w^*}(C)$. Then, for $z\in Z$, the following are equivalent:
		\begin{enumerate}[(i)]
			\item $u\perp_B z$;
			\item $0\in \conv\,\bigcap\nolimits_{\delta>0} \overline{\bigl\{\phi(z)\colon \phi\in C,\, \re \phi(u)>1-\delta\bigr\}}$;
			\item $0\in\conv \left(\bigl\{\lim \phi_n(z)\colon \phi_n\in C \ \forall n \in \N, \ \lim \phi_n(u)=1 \bigr\}\right)$.
		\end{enumerate}
	\end{corollary}
	
	\begin{corollary}\label{cor:B-J-orth-Z-Lambda}
		Let $Z$ be a Banach space, let $u\in S_Z$, and let $\Lambda\subset B_{Z^*}$ be one-norming for $Z$. Then, for $z\in Z$, the following are equivalent:
		\begin{enumerate}[(i)]
			\item $u\perp_B z$;
			\item $0\in\conv\left(\bigl\{\theta_0 \lim\psi_n(z)\colon \psi_n\in \Lambda \ \forall n\in \N,\ \theta_0\in\T, \ \lim\psi_n(u)=\overline{\theta_0}\bigr\}\right)$;
			\item $0\in\conv\left(\left\{\lim\psi_n(z)\overline{\psi_n(u)}\colon \psi_n\in \Lambda \ \forall n\in \N,\ \lim|\psi_n(u)|=1\right\}\right)$;
			\item  $0\in\conv\,\bigcap\nolimits_{\delta>0} \overline{\bigl\{\psi(z)\overline{\psi(u)} \colon \psi\in\Lambda, \ |\psi(u)|>1-\delta\bigr\}}$;
			\item $0\in\bigcap\nolimits_{\delta>0} \conv \overline{\bigl\{\psi(z)\overline{\psi(u)} \colon \psi\in\Lambda, \ |\psi(u)|>1-\delta\bigr\}}$.
		\end{enumerate}
	\end{corollary}
	
The next easy result will allow us to rephrase the characterizations given above without using the language of convex hull.

	\begin{lemma}\label{Convexity Lemma}
Let $A$ be a non-empty subset of $\K$. Then, $0\in \conv(A)$ if and only if given any $\mu\in \T$, there exists $a_\mu\in A$ such that $\re \mu a_\mu \geq 0$.

Moreover, if $A$ is connected, then $0\in \conv(A)$ if and only if given any $\mu\in \T$, there exists $a_\mu\in A$ such that $\re \mu a_\mu= 0$.
	\end{lemma}
	
The real case is obvious. In the complex case, the result follows straightforwardly from the Hahn-Banach separation theorem. An elementary proof of the sufficiency can be found  in \cite[Lemma~2.1]{RoyBagchi}. Let us give an elementary argument for the necessity. By Carath\'{e}odory's theorem, there exist $\lambda_j\geq 0$ and $a_j\in A$ $(j=1,2,3)$ such that $\sum\nolimits_{j=1}^3 \lambda_j=1$ and $\sum\nolimits_{j=1}^3 \lambda_ja_j=0$.
			Consider any $\mu\in \T$. Since $\sum_{j=1}^3 \mu\lambda_ja_j=0$, there exist $b_1, b_2\in \{a_1,a_2,a_3\}$ such that $\re \mu b_1 \geq 0$ and $\re \mu b_2 \leq 0$, and we are done. Now, if $A$ is connected, then $\re\{\mu a \colon a\in A\}$ is an interval in $\mathbb{R}$ containing $\re \mu b_1$ and $\re \mu b_2$. Thus, there exists $\widetilde{a}\in A$ such that $\re \mu \widetilde{a} =0$. This completes the argument.

Let us now use the same spirit of Corollaries \ref{cor:B-J-orth-Z-C} and \ref{cor:B-J-orth-Z-Lambda} to characterize the notion of smoothness in terms of the numerical range. Recall that a \emph{smooth point} $z$ of a Banach space $Z$ (we may also say that $Z$ is smooth at $z$) is just a point at which the norm of $Z$ is Gateaux differentiable; equivalently, $z$ is a smooth point if $\{\phi\in S_{Z^*}\colon \phi(z)=\|z\|\}$ is a singleton. The following lemma will allow us to use the characterizations of BJ-orthogonality to describe smooth points. Although the proof of the lemma is immediate, we record it for the sake of completeness.
	
	\begin{lemma}\label{lemma:smooth-V}
		Let $Z$ be a Banach space and let $u\in S_Z$. Then, $u$ is a smooth point if and only if $V(Z,u,z)$ is a singleton set for every $z\in Z$.
	\end{lemma}
	
	\begin{proof}
		If $u$ is smooth, then $\F(B_{Z^*},u)$ is a singleton and then so are the sets $V(Z,u,z)$ for all $z\in Z$. Conversely, suppose that there exist $\phi_1, \phi_2\in S_{Z^*}$ such that $\phi_1(u)=\phi_2(u)=1$ and $\phi_1\neq\phi_2$; then we may find $z\in S_Z$ such that $\phi_1(z)\neq\phi_2(z)$ and so $V(Z,u,z)$ is not a singleton set.
	\end{proof}
	
	The above lemma allows to characterize smoothness using Proposition~\ref{prop:num-range-ext}, Theorem~\ref{theorem:num-range-C}, and Theorem~\ref{thm:num-range-Lambda}.
	
	\begin{corollary}\label{corollary:charac-smooth-extreme}
		Let $Z$ be a Banach space and let $u\in S_Z$. Then, $u$ is a smooth point if and only if $\left\{\phi(z)\colon \phi\in\ext\left(B_{Z^*}\right), \ \phi(u)=1\right\}$ is a singleton set for every $z\in Z$.
	\end{corollary}
	
	\begin{corollary}\label{cor:characterization-smooth}
		Let $Z$ be a Banach space, let $u\in S_Z$, and let $C\subset B_{Z^*}$ be such that $B_{Z^*}=\overline{\conv}^{w^*}(C)$. Then, the following are equivalent:
		\begin{enumerate}[(i)]
			\item $u$ is a smooth point;
			\item $\bigl\{\lim \phi_n(z)\colon \phi_n\in C \ \forall n \in \N, \ \lim \phi_n(u)=1 \bigr\}$ is a singleton set for every $z\in Z$;
			\item $\bigcap\nolimits_{\delta>0} \overline{\bigl\{\phi(z)\colon \phi\in C,\, \re \phi(u)>1-\delta\bigr\}}$ is a singleton set for every $z\in Z$.
		\end{enumerate}
	\end{corollary}
	
	\begin{corollary}\label{cor:characterization-smooth-lambda}
		Let $Z$ be a Banach space, let $u\in S_Z$, and $\Lambda\subset B_{Z^*}$ be one-norming for $Z$. Then, the following are equivalent:
		\begin{enumerate}[(i)]
			\item $u$ is a smooth point;
			\item $\bigl\{\theta_0 \lim\psi_n(z)\colon \psi_n\in \Lambda \ \forall n\in \N,\ \theta_0\in\T, \ \lim\psi_n(u)=\overline{\theta_0}\bigr\}$ is a singleton set for every $z\in Z$;
			\item $\left\{\lim\psi_n(z)\overline{\psi_n(u)}\colon \psi_n\in \Lambda \ \forall n\in \N,\ \lim|\psi_n(u)|=1\right\}$ is a singleton set for every $z\in Z$;
			\item $\bigcap\nolimits_{\delta>0} \overline{\bigl\{\psi(z)\overline{\psi(u)} \colon \psi\in\Lambda, \ |\psi(u)|>1-\delta\bigr\}}$ is a singleton set for every $z\in Z$;
			\item $\bigcap\nolimits_{\delta>0} \conv \overline{\bigl\{\psi(z)\overline{\psi(u)} \colon \psi\in\Lambda, \ |\psi(u)|>1-\delta\bigr\}}$ is a singleton set for every $z\in Z$.
		\end{enumerate}
	\end{corollary}
	
	\section{Using the general results in some interesting particular cases}\label{section:someparticularcases}
	
	We devote this section to applying the abstract results of the previous section in several settings. We begin with a characterization of BJ-orthogonality  in a dual space that extends \cite[Theorem~2.7]{SainPaulMal-JOT2018} to the complex case. The proof is immediate from Corollary \ref{cor:B-J-orth-Z-C}, Corollary \ref{cor:characterization-smooth}, and Goldstine's theorem.
	
	\begin{proposition}\label{prop:orthogonality-smoothness-dualspace}
		Let $Y$ be a Banach space, let $C\subset B_Y$ such that $\conv(C)$ is dense in $B_Y$. For  $u^*, z^*\in Y^*$, we have that
		$$
		u^*\perp_B z^* \Longleftrightarrow 0\in\conv\bigl(\bigl\{\lim z^*(y_n)\colon y_n\in C \ \forall n\in\N, \ \lim u^*(y_n)=\|u^*\|\bigr\}\bigr).
		$$
		Moreover, the norm of $Y^*$ is smooth at $u^*$ if and only if the set
		$$
		\bigl\{\lim u^*(y_n)\colon y_n\in C \ \forall n\in\N, \ \lim u^*(y_n)=\|u^*\|\bigr\}
		$$
		is a singleton set for all $y^*\in Y^*$.
	\end{proposition}
	
	The results in the rest of the section are divided into subsections for clarity of the exposition.
	
	\subsection{Spaces of bounded functions}
	Given a non-empty set $\Gamma$ and a Banach space $Y$, we write $\ell_\infty(\Gamma,Y)$ to denote the Banach space of all bounded functions from $\Gamma$ to $Y$ endowed with the supremum norm. For $\gamma\in\Gamma$, $\delta_\gamma\colon \ell_\infty(\Gamma,Y)\longrightarrow Y$ denotes the evaluation map. We next characterize BJ-orthogonality  in $\ell_\infty(\Gamma,Y)$. Fix a subset $C\subset S_{Y^*}$ whose weak-star closed convex hull is the
whole $B_{Y^*}$. Consider the set
	$$
	\mathfrak{C}:=\{y^*\otimes\delta_\gamma\colon \gamma\in\Gamma, \ y^*\in C\}\subseteq \ell_\infty(\Gamma,Y)^*,
	$$
	where $[y^*\otimes \delta_\gamma](f):=y^*(f(\gamma))$ for every $f\in \ell_\infty(\Gamma,Y)$. Since, clearly,  $B_{\ell_\infty(\Gamma,Y)^*}$ is the weak-star closed convex hull of $\mathfrak{C}$, the following result is a consequence of Corollary~\ref{cor:B-J-orth-Z-C}.
	
	\begin{theorem}\label{theorem:ellinftygammaY}
		Let $\Gamma$ be a non-empty set, let $Y$ be a Banach space, let $C\subset S_{Y^*}$ be such that $B_{Y^*}=\overline{\conv}^{w^*}(C)$, and let $f,g\in \ell_\infty(\Gamma,Y)$. Then,
		$$
		f\perp_B g \Longleftrightarrow 0\in\conv\left\{\lim y_n^*(g(\gamma_n))\colon \gamma_n\in\Gamma, \ y_n^*\in C \ \forall n\in \N, \ \lim y_n^*(f(\gamma_n))=\|f\|\right\}.
		$$
	\end{theorem}
	
	Of course, the same characterization is valid in every closed subspace of $\ell_\infty(\Gamma,Y)$, since the BJ-orthogonality only depends on the two-dimensional subspace generated by the involved vectors. Then, as a consequence, we get a characterization of smoothness in any closed subspace $\mathcal{Z}\subseteq \ell_\infty(\Gamma,Y)$.
	
	\begin{corollary}\label{coro:ellinftygammaY-smooth}
		Let $\Gamma$ be a non-empty set, let $Y$ be a Banach space, let $C\subset S_{Y^*}$ such that $B_{Y^*}=\overline{\conv}^{w^*}(C)$, and let $\mathcal{Z}\subseteq \ell_\infty(\Gamma,Y)$ be a closed subspace. Then, for $f\in \mathcal{Z}$ the following are equivalent:
		\begin{enumerate}[(i)]
			\item $f$ is a smooth point;
			\item $\bigl\{\lim y_n^*(g(\gamma_n))\colon \gamma_n\in\Gamma, \ y_n^*\in C \ \forall n\in \N, \ \lim y_n^*(f(\gamma_n))=\|f\|\bigr\}$ is a singleton set for every $g\in \mathcal{Z}$.
		\end{enumerate}
	\end{corollary}
	
	As far as we know, the above two results are new.
	
	Let us consider some interesting particular cases. Given a Hausdorff topological space $\Omega$ and a Banach space $Y$, we write $C_b(\Omega,Y)$ to denote the Banach space of all bounded continuous functions from $\Omega$ to $Y$, endowed with the supremum norm.
	
	\begin{corollary}\label{corollary (to be used)}
		Let $\Omega$ be a Hausdorff topological space, let $Y$ be a Banach space, and let $f,g\in C_b(\Omega,Y)$. Then,
		$$
		f\perp_B g \Longleftrightarrow 0\in\conv\left\{\lim  y_n^*(g(t_n))\colon t_n\in\Omega, \ y_n^*\in S_{Y^*} \ \forall n\in \N, \ \lim y_n^*(f(t_n))=\|f\|\right\}.
		$$
	\end{corollary}
	
	This result extends \cite[Corollary~3.1]{Keckic-2012} to the vector-valued case. When $\Omega$ is compact, the result can be improved using Fact~\ref{fact:Singer} and the description of the dual ball of $C(K,Y)=\mathcal{K}_{w^*}(Y^*,C(K))$ given in \cite[Theorem~1.1]{Ruess-Stegall-MathAnn1982}.
	
	\begin{corollary}\label{corollary:C(K,Y)extreme}
		Let $K$ be a compact Hausdorff topological space, let $Y$ be a Banach space, and let $f,g\in C(K,Y)$. Then,
		$$
		f\perp_B g \Longleftrightarrow 0\in\conv\left\{y^*(g(t))\colon t\in K, \, y^*\in \ext(B_{Y^*}),\, y^*(f(t))=\|f\|\right\}.
		$$
		Moreover, $f\in C(K,Y)$ is smooth if and only if the set
		$$
		\left\{y^*(g(t))\colon t\in K, \, y^*\in \ext(B_{Y^*}),\, y^*(f(t))=\|f\|\right\}
		$$
		is a singleton for every $g\in C(K,Y)$.
	\end{corollary}
	
The first part of the above corollary improves \cite[Theorem~2.1]{RoySenapatiSain}, where the result was given only in the real case, and \cite[Theorem 2.2]{RoyBagchi}, where it was proved in the case when $Y$ is a finite-dimensional Hilbert space.

	Another case in which Theorem~\ref{theorem:ellinftygammaY} applies is the one of \emph{unital uniform algebras}: closed subalgebras of a $C(K)$ space separating the points of $K$ and containing the constant functions. Actually, in this case an improved result can be stated. For a unital uniform algebra $A$ on $C(K)$, the \emph{Choquet boundary} of $A$ is the set
	$$
	\partial A:=\{s\in K\colon \delta_s|_A\in \ext(B_{A^*})\}
	$$
	endowed with the topology induced by $K$. We refer to \cite[Chap.~6]{Phelps-Choquet} for background. It is immediate that
	$$
	\ext(B_{A^*})=\T\{\delta_s|_A\colon s\in \partial A\},
	$$
	hence the next result follows from Fact~\ref{fact:Singer} and Corollary~\ref{corollary:charac-smooth-extreme}.
	
	\begin{corollary}\label{corollary:uniformalgebras}
		Let $A$ be a unital uniform algebra on $C(K)$ and let $\partial A\subset K$ be its Choquet boundary.
		\begin{enumerate}
			\item[(a)] $f,g\in A$ satisfy $f\perp_B g$ if and only if
			$$
			0\in \conv\bigl\{\theta g(s)\colon \theta\in\T,\, s\in\partial A,\, f(s)=\overline{\theta}\|f\|\bigr\}
			$$
			\item[(b)] $f\in A$ is a smooth point of $A$ if and only if the set
			$$
			\bigl\{\theta g(s)\colon \theta\in\T,\, s\in\partial A,\, f(s)=\overline{\theta}\|f\|\bigr\}
			$$
			is a singleton for every $g\in A$.
		\end{enumerate}
	\end{corollary}
	
	This result applies, in particular, to the \emph{disk algebra} $\mathbb{A}(\mathbb{D})$ of those continuous functions on the unit disk $\mathbb{D}=\{w\in \C\colon |w|\leq 1\}$ which are holomorphic in the interior, whose Choquet boundary is $\T$. We will improve this result for a certaing class of holomorphic functions on the open unit disk in Corollary~\ref{corollary:holomorphic}.

	\subsection{Lipschitz maps}
	
	Next, we give a characterization of BJ-orthogonality  in the space of Lipschitz maps. To do so, we present the basic notions and notations. Given a pointed metric space (that is, a metric space $M$ with a distinguished element called $0$) and a Banach space $Y$, we denote by $\Lip_{0}(M,Y)$ the Banach space of all Lipschitz maps $F\colon  M \longrightarrow Y$ such that $F(0)=0$ endowed with the norm
	$$
	\| F\|_L = \sup \left\{ \frac{\|F(t) - F(s)\|}{d(t,s)}\colon t,s\in M,\ t \neq s \right\}.
	$$
	We refer the reader to the book \cite{Weaver} for more information and background.
	Given $s,t\in M$, $s\neq t$, and $y^*\in Y^*$, we define
	$$
	\left[\widetilde{\delta}_{s,t}\otimes y^*\right](F):=\frac{y^*(F(t) - F(s))}{d(t,s)}
	$$
	for every $F\in \Lip_0(M,Y)$. It is immediate that this formula defines a bounded linear functional on $\Lip_0(M,Y)$ and that, given a one-norming subset $C\subset S_{Y^*}$ for $Y$, the subset
	$$
	\mathfrak{C}:=\left\{\widetilde{\delta}_{s,t}\otimes y^*\colon s,t\in M,\, s\neq t,\, y^*\in C\right\}
	$$
	is one-norming for $\Lip_0(M,Y)$. Therefore, Corollary~\ref{cor:B-J-orth-Z-Lambda} and Corollary~\ref{cor:characterization-smooth-lambda} give the following result.
	
	\begin{proposition}\label{prop:Lipschitz}
		Let $M$ be a pointed metric space, let $Y$ be a Banach space, and let $C\subseteq S_{Y^*}$ be one-norming for $Y$.
		\begin{enumerate}
			\item[(a)] $F,G\in \Lip_0(M,Y)$ satisfy $F\perp _B G$ if and only if $0$ belongs to $$\conv \left\{\lim \textstyle{\theta_0\frac{y_n^*\bigl(G(s_n)-G(t_n)\bigr)}{d(s_n,t_n)}} \colon s_n,t_n\in M, \, s_n\neq t_n, \, y_n^*\in C, \, \theta_0\in\T,\, \lim \textstyle{\frac{y_n^*\bigl(F(s_n)-F(t_n)\bigr)}{d(s_n,t_n)}}= \overline{\theta_0}\|F\|_L\right\}.$$
			\item[(b)] $F\in \Lip_0(M,Y)$ is a smooth point if and only if the set $$\left\{\lim \textstyle{\theta_0\frac{y_n^*\bigl(G(s_n)-G(t_n)\bigr)}{d(s_n,t_n)}} \colon s_n,t_n\in M, \, s_n\neq t_n, \, y_n^*\in C, \, \theta_0\in\T,\, \lim \textstyle{\frac{y_n^*\bigl(F(s_n)-F(t_n)\bigr)}{d(s_n,t_n)}}= \overline{\theta_0}\|F\|_L\right\}$$
			is a singleton for every $G\in \Lip_0(M,Y)$.
		\end{enumerate}
	\end{proposition}
	
	Let us comment that there is a result on smoothness in spaces of Lipschitz functions showing that smoothness and Fr\'{e}chet smoothness are equivalent in $\Lip_0(M,\R)$, see \cite[Corollary~5.8]{GL-Pet-Pro-RZ-2018}.
	
	This result also follows from Theorem~\ref{theorem:ellinftygammaY} by using a vector-valued version of De Leeuw's map, see \cite[\S 2.4]{Weaver} for instance.
	
	\subsection{Injective tensor products}
	
	Let $X$, $Y$ be Banach spaces.
	The \textit{injective tensor product} of $X$ and $Y$, denoted by $X\hat{\otimes}_\eps Y$, is the completion of $X\otimes Y$ endowed with the norm given by
	$$
	\|u\|_\eps=\sup\left\{ \left|\sum_{i=1}^n x^*(x_i)y^*(y_i)\right|\colon x^*\in B_{X^*},\, y^*\in B_{Y^*} \right\},
	$$
	where $\sum_{i=1}^n x_i\otimes y_i$ is any representation of $u$. Since $B_{(X\hat{\otimes}_\eps Y)^*}= \overline{\conv}^{w^*}(B_{X^*}\otimes B_{Y^*})$, we obtain the following result as consequence of Corollaries~\ref{cor:B-J-orth-Z-C} and \ref{cor:characterization-smooth}.
	
	\begin{proposition}\label{prop:B-J-orth-injective-tensor}
		Let $X$, $Y$ be Banach spaces.
		\begin{enumerate}[(a)]
			\item $u,z\in X\hat{\otimes}_\eps Y$ satisfy $u\perp_B z$ if and only if
			$$0\in\conv\bigr(\bigl\{\lim (x_n^*\otimes y_n^*)(z)\colon x_n^*\otimes y_n^*\in B_{X^*}\otimes B_{Y^*} \, \forall n\in\N, \ \lim(x_n^*\otimes y_n^*)(u)=\|u\|_\eps \bigr\}\bigl).
			$$
			\item $u\in X\hat{\otimes}_\eps Y$ is smooth if and only if the set
			$$\bigl\{\lim (x_n^*\otimes y_n^*)(z)\colon x_n^*\otimes y_n^*\in B_{X^*}\otimes B_{Y^*} \, \forall n\in\N, \ \lim(x_n^*\otimes y_n^*)(u)=\|u\|_\eps \bigr\}$$
			is a singleton for every $z\in X\hat{\otimes}_\eps Y$.
		\end{enumerate}
	\end{proposition}

	\subsection{Spaces of operators endowed with the operator norm}\label{subsection:operators}
	Let $X$, $Y$ be Banach spaces. Consider $C\subset S_X$ such that $\conv(C)$ is dense in $B_X$ and $D\subset S_{Y^*}$ such that $\conv(D)$ is weak-star dense in $B_{Y^*}$. Our general characterization of BJ-orthogonality  in $\mathcal{L}(X,Y)$ endowed with the usual norm is obtained by using Corollary~\ref{cor:B-J-orth-Z-C} with
	$$\mathfrak{C}:=\{y^*\otimes x \colon x\in C, \ y^*\in D\}
	$$
	where $y^*\otimes x\in\mathcal{L}(X,Y)^*$ is defined by
	$$
	[y^*\otimes x](T):=y^*(Tx) \qquad (T\in\mathcal{L}(X,Y)).
	$$
	For $C=S_X$ and $D=S_{Y^*}$, the result already appeared in \cite[Theorem~2.2]{Roypre2022}, with a different proof.
	
	\begin{proposition}[\mbox{Extension of \cite[Theorem~2.2]{Roypre2022}}]\label{prop:B-J-orth-operator-norm}
		Let $X$, $Y$ be Banach spaces, $C\subset S_X$ such that $\conv(C)$ is dense in $B_X$ and $D\subset S_{Y^*}$ such that $\conv(D)$ is weak-star dense in $B_{Y^*}$, and let $T,A\in\mathcal{L}(X,Y)$. Then,
		$$
		T\perp_B A \Longleftrightarrow 0\in \conv\left(\bigl\{\lim y_n^*(Ax_n)\colon (x_n,y_n^*)\in C \times D \ \forall n\in\N, \ \lim y_n^*(Tx_n)=\|T\|\bigr\}\right).
		$$
	\end{proposition}
	
	Observe that the result also follows from Theorem~\ref{theorem:ellinftygammaY} as $\mathcal{L}(X,Y)$ can be viewed as a closed subspace of $\ell_\infty(C,Y)$.
	
	When the operators involved are compact we can remove the limits in Proposition~\ref{prop:B-J-orth-operator-norm} and also we can use extreme points of $B_{X^{**}}$ and of $B_{Y^*}$. For $y^*\in Y^*$ and $x^{**}\in X^{**}$, we consider  $[x^{**}\otimes y^*](T):=x^{**}(T^*y^*)$ for every $T\in \mathcal{K}(X,Y)$.
	
	\begin{proposition}\label{prop:compact-operators-general-extreme}
		Let $X$, $Y$ be Banach spaces, and let $T,A\in\mathcal{K}(X,Y)$. Then,
		$$
		T\perp_B A \Longleftrightarrow 0\in \conv\left(\bigl\{x^{**}(A^*(y^*))\colon x^{**}\in \ext(B_{X^{**}}), \ y^*\in \ext(B_{Y^*}), \ x^{**}(T^{*}(y^*))=\|T\|\bigr\}\right).
		$$
	\end{proposition}

	The proof of this result follows from Fact~\ref{fact:Singer} as the set $$C=\{x^{**}\otimes y^* \colon x^{**}\in \ext(B_{X^{**}}), \ y^*\in \ext(B_{Y^*})\}$$ coincides with the set of extreme points of the unit ball of $\mathcal{K}(X,Y)^*$ (see \cite[Theorem 1.3]{Ruess-Stegall-MathAnn1982} for the real case and \cite[Theorem 1]{LimaOlsen} for the complex case). In the case when $X$ is reflexive, the above result has a nicer form. Let us remark here that a special case of the following result was obtained in Theorem $ 2.1 $ of \cite{SainPaulMal-JOT2018}, where $ X $ is assumed to be a real reflexive Banach space.
	
	\begin{corollary}\label{B-J-orth-operator-norm-compact}
		Let $X$ be a reflexive Banach space, let $Y$ be a Banach space, and let $T,A\in\mathcal{K}(X,Y)$. Then,
		$$
		T\perp_B A \Longleftrightarrow 0\in \conv\left(\bigl\{y^*(Ax)\colon x\in \ext(B_X), \ y^*\in \ext(B_{Y^*}), \ y^*(Tx)=\|T\|\bigr\}\right).
		$$
	\end{corollary}
	
	Of course, the previous result applies when $X$ is finite-dimensional.
	
	\begin{corollary}[\mbox{\cite[Proposition~4.2]{LiSchneider}}]\label{B-J-orth-operator-norm-finite}
		Let $X$ be a finite-dimensional space, let $Y$ be a Banach space, and let $T,A\in\mathcal{L}(X,Y)$. Then
		$$
		T\perp_B A \Longleftrightarrow 0\in \conv\left(\bigl\{y^*(Ax)\colon x\in \ext(B_X), \ y^*\in \ext(B_{Y^*}), \ y^*(Tx)=\|T\|\bigr\}\right).
		$$
	\end{corollary}
	
	We finish this subsection on the operator norm by presenting a characterization of smooth operators which follows directly from Corollary~\ref{cor:characterization-smooth}.
	
	\begin{proposition}\label{prop:characterization-smooth-operator}
		Let $X$, $Y$ be Banach spaces, $C\subset S_X$ such that $\conv(C)$ is dense in $B_X$ and $D\subset S_{Y^*}$ such that $\conv(D)$ is weak-star dense in $B_{Y^*}$, and let $0\neq T\in\mathcal{L}(X,Y)$. Then, $T$ is a smooth operator if and only if
		$$
		\bigl\{\lim y_n^*(Ax_n)\colon (x_n,y_n^*)\in C \times D \ \forall n\in\N, \ \lim y_n^*(Tx_n)=\|T\|\bigr\}
		$$
		is a singleton set for every $A\in\mathcal{L}(X,Y)$.
	\end{proposition}
	
	As an easy consequence of this proposition and   Corollary~\ref{cor:characterization-smooth}, we obtain a result that gives the existence of smooth operators under reasonable restrictions. In fact, they are quite similar to those used by Heinrich in \cite[Theorem~3.1]{Heinrich} to characterize Fr\'{e}chet smooth operators in $\mathcal{L}(X,Y)$ (but there is no characterization of smoothness of operators outside $\mathcal{K}(X,Y)$ in \cite{Heinrich}). The result extends \cite[Theorem~3.4]{SainPaulMalRay} to the complex case. It will be used in Subsection~\ref{subsec:obstructive-spearoperators}.
	
	\begin{proposition}\label{prop:suff-condition-smooth-operator}
		Let $X$, $Y$ be Banach spaces. Let $0\neq T\in \mathcal{L}(X,Y)$ be such that there is $x_0\in S_X$ satisfying the
		following conditions:
		\begin{enumerate}
			\item[(1)] $T x_0$ is a smooth point in $Y$;
			\item[(2)] every sequence $\{x_n\}\subset B_X$ satisfying $\lim \|Tx_n\|=\|T\|$  has a subsequence converging to $\alpha x_0$ for some $\alpha\in \T$.
		\end{enumerate}
		Then, $T$ is smooth.
	\end{proposition}
		
	\begin{proof}
		Using Proposition~\ref{prop:characterization-smooth-operator} it suffices to show that, for every $A\in \mathcal{L}(X,Y)$, the set
		$$
		\bigl\{\lim y_n^*(Ax_n)\colon (x_n,y_n^*)\in S_X \times S_{Y^*} \ \forall n\in\N, \ \lim y_n^*(Tx_n)=\|T\|\bigr\}
		$$
		is a singleton. To do so, fix an arbitrary $\lambda=\lim y_n^*(Ax_n)$ and observe that $\lim y_n^*(Tx_n)=\|T\|$ implies $\lim\|Tx_n\|=\|T\|$. So, using (2), there are $\alpha\in \T$ and a subsequence $\{x_{\sigma(n)}\}$ with $\lim x_{\sigma(n)}=\alpha x_0$. Now, it is clear that
		$$
		\lim(\alpha y_{\sigma(n)}^*)(Tx_0)=\lim y_{\sigma(n)}^*(Tx_{\sigma(n)})=\lim y_n^*(Tx_n)=\|T\|=\|Tx_0\|
		$$
		and
		$$
		\lim(\alpha y_{\sigma(n)}^*)(Ax_0)=\lim y_{\sigma(n)}^*(Ax_{\sigma(n)})=\lambda.
		$$
		Therefore, we get that
		$$
		\lambda\in \bigl\{\lim z_n^*(Ax_0)\colon z_n^*\in S_{Y^*} \ \forall n\in\N, \ \lim z_n^*(Tx_0)=\|Tx_0\|\bigr\}
		$$
		and the latter set is a singleton by Corollary~\ref{cor:characterization-smooth} as $Tx_0$ is a smooth point of $Y$ by (1).
	\end{proof}

	\subsection{Multilinear maps and polynomials}\label{subsection:multilinear-polynomials}
	In an analogous way that we deal with bounded operators, it is possible to describe the BJ-orthogonality of multilinear maps and polynomials.
	
	Let $X_1, \ldots, X_k$ and $Y$ be Banach spaces. The set of all bounded $k$-linear maps from $X_1 \times \cdots \times X_k$ to $Y$ will be denoted by $\mathcal{L}(X_1, \ldots,X_k;Y)$. As usual, we define the norm of $A \in \mathcal{L}(X_1, \ldots, X_k; Y)$ by
	\[
	\|A\| = \sup \bigl\{ \|A(x_1, \ldots, x_k) \| \colon (x_1, \ldots, x_k) \in S_{X_1} \times \cdots \times S_{X_k} \bigr\}.
	\]
	It is then immediate that $$\mathcal{L}(X_1, \ldots,X_k;Y)\subset \ell_\infty(\Gamma,Y)$$ where $\Gamma=S_{X_1} \times \cdots \times S_{X_k}$. Therefore, the following result follows immediately from Theorem~\ref{theorem:ellinftygammaY} and Corollary~\ref{coro:ellinftygammaY-smooth}. It was proved in \cite{Roypre2022}.
	
	\begin{proposition}[\mbox{\cite[Theorem~2.2 and Theorem~3.1]{Roypre2022}}]
		Let $X_1, \ldots, X_k$ and $Y$ be Banach spaces and let $C\subset S_{Y^*}$ be such that $B_{Y^*}=\overline{\conv}^{w^*}(C)$.
		\begin{enumerate}[(a)]
			\item For $T,A\in \mathcal{L}(X_1, \ldots,X_k;Y)$ we have that $T\perp_B A$ if and only if $0$ belongs to the convex hull of
			$$
			\left\{\lim_n y_n^*(A(x^n_1,\ldots,x^n_k))\colon (x^n_1, \ldots, x^n_k) \in S_{X_1} \times \cdots \times S_{X_k},\, y_n^*\in C,\,  \lim_n y_n^*(T(x^n_1,\ldots,x^n_k))=\|T\|\right\}.
			$$
			\item $T\in \mathcal{L}(X_1, \ldots,X_k;Y)$ is a smooth point if and only if the set
			$$
			\left\{\lim_n y_n^*(A(x^n_1,\ldots,x^n_k))\colon (x^n_1, \ldots,x^n_k) \in S_{X_1} \times \cdots \times S_{X_k},\, y_n^*\in C,\,  \lim_n y_n^*(T(x^n_1,\ldots,x^n_k))=\|T\|\right\}
			$$
			is a singleton for every $A\in \mathcal{L}(X_1, \ldots,X_k;Y)$.
		\end{enumerate}
	\end{proposition}

	We now deal with polynomials between Banach spaces. Let $X$ and $Y$ be Banach spaces. A (continuous) \emph{$N$-homogeneous polynomial} $P$ from $X$ to $Y$ is a mapping $P\colon X \longrightarrow Y$ for which we can find a multilinear operator $T\in \mathcal{L}(X\times \overbrace{\ldots}^N \times X;Y)$ (continuous) which is symmetric (i.e., $T(x_1,\ldots, x_N) = T(x_{\sigma(1)},\ldots, x_{\sigma(N)})$ for every permutation $\sigma$ of the set $\{1,\ldots, N\}$) and satisfying $P(x) = T(x , \ldots , x)$ for every  $x \in X$.
	A (general) polynomial from $X$ to $Y$ is a mapping $P\colon X\longrightarrow Y$ which can be written as a finite sum of homogeneous polynomials. We write $\mathcal{P}(X,Y)$ for the space of all polynomials from $X$ to $Y$. It is immediate that $\mathcal{P}(X,Y)$ is a subspace of $\ell_\infty(B_X,Y)$, so the next result follows again from Theorem~\ref{theorem:ellinftygammaY} and Corollary~\ref{coro:ellinftygammaY-smooth}.
	
	\begin{proposition}\label{prop:polynomials}
		Let $X$, $Y$ be Banach spaces and let $C\subset S_{Y^*}$ be such that $B_{Y^*}=\overline{\conv}^{w^*}(C)$.
		\begin{enumerate}[(a)]
			\item Given $P,Q\in \mathcal{P}(X,Y)$, we have that $P\perp_B Q$ if and only if
			$$
			0\in \conv\left\{\lim  y_n^*(P(x_n))\colon x_n\in B_X,\, y_n^*\in C,\, \lim y_n^*(Q(x_n))=\|Q\|\right\}.
			$$
			\item $P\in \mathcal{P}(X,Y)$ is a smooth point if and only if the set
			$$
			\left\{\lim  y_n^*(P(x_n))\colon x_n\in B_X,\, y_n^*\in C,\, \lim y_n^*(Q(x_n))=\|Q\|\right\}
			$$
			is a singleton for every $Q\in \mathcal{P}(X,Y)$.
		\end{enumerate}
	\end{proposition}

	\subsection{Spaces of operators endowed with the numerical radius as norm}\label{subsection:numericalradiusnorm}
	Let $X$ be a Banach space. We deal here with the space $\mathcal{L}(X)$ endowed with the numerical radius. Let us recall the necessary definitions. Write $\Pi(X):=\{(x,x^*)\in S_X\times S_{X^*}\colon x^*(x)=1 \}$. The \emph{numerical radius} of $T\in\mathcal{L}(X)$ is
	\begin{align*}
		v(T):=\sup\{|x^*(Tx)|\colon (x,x^*)\in \Pi(X)\}.
	\end{align*}
	It is a well-known fact that
	$$
	v(T)=\sup\bigl\{|\lambda|\colon  \lambda\in V(\mathcal{L}(X),\Id,T)\bigr\}
	$$
	for every $T\in\mathcal{L}(X)$ (see \cite[Proposition 2.1.31]{Cabrera-Rodriguez}, for instance). We refer the interested reader to the classical books \cite{BonsallDuncan1,B-D2} and to Sections 2.1 and 2.9 of the book \cite{Cabrera-Rodriguez} for more information and background. It is clear that the numerical radius is a seminorm on $\mathcal{L}(X)$ and $v(T)\leq\|T\|$ for every $T\in\mathcal{L}(X)$. We would like to remark here that although BJ-orthogonality is defined in the framework of norms, it may also be considered in exactly the same way in any seminormed space. Of course, when the seminorm is a norm, we return to the original setting.

	We particularize Corollary~\ref{cor:B-J-orth-Z-Lambda} to the space of operators with the numerical radius, taking
	$$
	\Lambda:=\{x^*\otimes x\colon (x,x^*)\in\Pi(X)\}\subset (\mathcal{L}(X),v)^*
	$$
	which is clearly one-norming for $(\mathcal{L}(X),v)$. The following result appeared in \cite[Theorem~3.4]{Mal}.
	
	\begin{proposition}[\mbox{\cite[Theorem~3.4]{Mal}}]
		Let $X$ be a Banach space and let $T,A\in\mathcal{L}(X)$. Then,
		$$
		T\perp_B^v A \Longleftrightarrow 0\in \conv\left(\bigl\{\lim x_n^*(Ax_n)\overline{x_n^*(T x_n)}\colon (x_n,x_n^*)\in \Pi(X) \ \forall n\in\N, \ \lim |x_n^*(Tx_n)|=v(T)\bigr\}\right).
		$$
	\end{proposition}
	
	In the case of compact operators defined on a reflexive space, it is straightforward to show that the limits can be removed.
	
	\begin{corollary}\label{cor:B-J-orth-num-radius-compact}
		Let $X$ be a reflexive Banach space and let $T,A\in\mathcal{K}(X)$. Then,
		$$
		T\perp_B^v A \Longleftrightarrow 0\in \conv\left(\bigl\{x^*(Ax)\overline{x^*(Tx)}\colon (x,x^*)\in\Pi(X), \ |x^*(Tx)|=v(T)\bigr\}\right).
		$$
	\end{corollary}
	
	A characterization in the particular case when $X$ has finite dimension has recently been proved by Roy and Sain \cite[Theorem~2.3]{RoySain}.

	\begin{corollary}[{\cite[Theorem~2.3]{RoySain}}]
		Let $X$ be a finite-dimensional space and let $T,A\in\mathcal{L}(X)$. Then
		$$
		T\perp_B^v A \Longleftrightarrow 0\in \conv\left(\bigl\{x^*(Ax)\overline{x^*(Tx)}\colon (x,x^*)\in\Pi(X), \ |x^*(Tx)|=v(T)\bigr\}\right).
		$$
	\end{corollary}
	
	We may state the following characterization of smoothness in $(\mathcal{L}(X),v),$ as a consequence of the previous observations and Corollary~\ref{cor:characterization-smooth-lambda}. We say that $T\in \mathcal{L}(X)$ is a \emph{smooth operator for the numerical radius} if $T$ is a smooth point of $(\mathcal{L}(X),v)$.
	
	\begin{proposition}
		Let $X$ be a Banach space and let $T\in\mathcal{L}(X)$. Then, $T$ is a smooth operator for the numerical radius if and only if
		$$
		\bigl\{\lim x_n^*(Ax_n)\overline{x_n^*(T x_n)}\colon (x_n,x_n^*)\in \Pi(X) \ \forall n\in\N, \ \lim |x_n^*(Tx_n)|=v(T)\bigr\}
		$$
		is a singleton set for every $A\in\mathcal{L}(X)$.
	\end{proposition}
	
As far as we could check, the above characterization of smoothness for the numerical radius has not appeared previously in its most general form.

	\section{Bhatia-\v{S}emrl's kind of results}\label{section:BS-kindofresults}
	In the particular case of operators on Hilbert spaces, the results of the Subsection~\ref{subsection:operators} can be improved as there is no need of taking convex hull.  The first characterization in this line was obtained by Stampfli \cite[Theorem~2]{Stampfli} in the special case when one of the operators is the identity. Later, Magajna \cite[Lemma~2.2]{Magajna} observed that Stampfli's result holds for any pair of operators, leading to a complete characterization of BJ-orthogonality  in $\mathcal{L}(H)$. The same characterization was obtained by Bhatia and \v{S}emrl \cite[Remark~3.1]{BhatiaSmerl}, and also by Ke\v{c}ki\'{c} \cite[Corollary~3.1]{Keckic-2005} with different approaches.
	
	Here we present an alternative proof which follows from our Proposition~\ref{prop:B-J-orth-operator-norm} and \cite[Theorem 2]{PaulHosseinDas}.
	
	\begin{corollary}[\mbox{\cite[Lemma~2.2]{Magajna}, \cite[Remark~3.1]{BhatiaSmerl}, \cite[Corollary~3.1]{Keckic-2005}}]
		Let $H$ be a Hilbert space and let $T,A\in\mathcal{L}(H)$. Then
		$ T\perp_B A$ if and only if there exists a sequence $\{x_n\}_{n\in\N}$ in $S_H$ such that $\|Tx_n\|\to\|T\|$ and $\langle Tx_n,Ax_n\rangle\to 0$.
	\end{corollary}
	
	\begin{proof}
		It follows from Proposition~\ref{prop:B-J-orth-operator-norm} that
		$$ T\perp_B A \Longleftrightarrow 0\in \conv\left(\bigl\{\lim \langle Ax_n,y_n\rangle\colon x_n,y_n\in S_H \ \forall n\in\N, \ \lim \langle Tx_n,y_n\rangle=\|T\|\bigr\}\right).$$
		Observe that
		\begin{align*}
			\bigl\{\lim \langle Ax_n,y_n\rangle&\colon x_n,y_n\in S_H \ \forall n\in\N, \ \lim \langle Tx_n,y_n\rangle=\|T\|\bigr\} \\
			&=\{\lim \langle Ax_n,Tx_n\rangle\colon x_n\in S_H \ \forall n\in\N, \ \lim\|Tx_n\|=\|T\|\bigr\}
		\end{align*}
		and that the latter set is convex (this was first stated without proof in \cite[Lemma 2.1]{Magajna}, see \cite[Theorem 2]{PaulHosseinDas} for a proof).
	\end{proof}
	
	In the particular case when $H$ is finite-dimensional, Bhatia and \v{S}emrl were the first to write down the characterization of BJ-orthogonality  of two matrices in terms of the elements of $H$ \cite[Theorem~1.1]{BhatiaSmerl}. An alternative proof of this characterization was given by Roy, Bagchi, and Sain in \cite{RoyBagchiSain}. We obtain this result as a consequence of Corollary~\ref{B-J-orth-operator-norm-finite}.
	
	\begin{corollary}[\mbox{Bhatia-\v{S}emrl theorem, \cite[Theorem~1.1]{BhatiaSmerl}}]\label{corollary:BhatiaSmerltheorem}
		Let $H$ be a finite-dimensional Hilbert space and let $T,A\in\mathcal{L}(H)$. Then
		$ T\perp_B A$ if and only if there exists $x\in S_H$ such that $\|Tx\|=\|T\|$ and $Tx\perp_B Ax$.
	\end{corollary}
	
	\begin{proof}
		We may and do suppose that $\|T\|=1$. It follows from Corollary~\ref{B-J-orth-operator-norm-finite} that
		$$
		T\perp_B A \Longleftrightarrow 0\in \conv\left(\bigl\{\langle Ax,y\rangle\colon x, y\in S_H, \ \langle Tx,y\rangle=\|T\|=1\bigr\}\right).
		$$
		Now, observe that
		$$
		\bigl\{\langle Ax,y\rangle\colon x, y\in S_H, \ \langle Tx,y\rangle=1\bigr\}=\bigl\{\langle Ax, Tx\rangle\colon x\in S_H, \ \|Tx\|=1\bigr\}.
		$$
		The result follows since the latter set is convex (\cite[Lemma 2.1]{Magajna}, \cite[Theorem 2]{PaulHosseinDas}).
	\end{proof}
	
	It has been shown by Li and Schneider that the Bhatia-\v{S}emrl theorem cannot be extended in general to arbitrary finite-dimensional Banach spaces \cite[Example~4.3]{LiSchneider}. Actually, the validity of the Bhatia-\v{S}emrl theorem for all operators characterizes Hilbert spaces among finite-dimensional Banach spaces, see \cite{BenitezFernandezSoriano}. However, it is natural to study for which operators $T$ it is possible to have a Bhatia-\v{S}emrl theorem for all operators $A$: conditions on $T$ such that whenever $T\perp_B A$, one has that there is a norm-one $x$ such that $\|Tx\|=\|T\|$ and $Tx\perp_B Ax$ (that is, whether we may remove the convex hull in Corollary~\ref{B-J-orth-operator-norm-compact}). This has been done in \cite{PaulSainGhosh-2016, Sain-JMAA-2017, SainPaul-2013} for the real case and in \cite{PaulSainMalMandal, RoyBagchiSain} for the complex case. Our aim in what follows is to give a unified approach that allows to recover some of these results and to obtain an improvement in the complex setting. Actually, we will work in the more general framework of vector-valued continuous functions on a compact Hausdorff space. To deal with both the real and the complex case, we need to introduce the notion of directional orthogonality from \cite{PaulSainMalMandal}. Given elements $x,y$ of a Banach space $Z$, we say that $x$ is \emph{orthogonal} to $y$ \emph{in the direction of $\gamma\in\T$}, which we denote by $x\perp_\gamma y$, if $\|x+t\gamma y\|\geq \|x\|$ for every $t\in\R$. Obviously, $x\perp_B y$ if and only if $x\perp_\gamma y$ for every $\gamma\in \T$. In the real case, it is obvious that $x\perp_B y$ if and only if $x\perp_{1} y$ if and only if $x\perp_{-1} y$. In the complex case, there are easy examples showing that $x\not\perp_B y$ while $x\perp_\gamma y$ for some $\gamma\in \T$ is possible, see \cite[Example~1]{RoyBagchiSain}. It is shown in \cite[Theorem~4]{RoyBagchiSain} that
		\begin{equation}\label{equation:directionalortho}
		x\perp_\gamma y \Longleftrightarrow \exists\ x^*\in S_{X^*} \text{ with } x^*(x)=\gamma\|x\| \text{ and } \re x^*(y)=0
		\end{equation}
(indeed, this result is immediate as $x\perp_\gamma y$ if and only if $x\perp_B \gamma y$ in the real space $X_\R$ underlying $X$ and $(X_\R)^*=\{\re x^*\colon x^*\in X^*\}$).

For a Hausdorff compact topological space $K$ and a Banach space $Y$, the norm attainment set of $f\in C(K,Y)$ is the (non-empty) set
$$
\mathcal{M}_f=\{t\in K \colon \|f(t)\|=\|f\|\}.
$$
Our main result in $C(K,Y)$ is a Bhatia-\v{S}emrl's type result when $\mathcal{M}_f$ is connected.
	
	\begin{theorem}\label{Compact continuous}
Let $K$ be a compact Hausdorff topological space and let $Y$ be a Banach space. Let $f,g\in C(K,Y)$ be such that $\mathcal{M}_f$ is connected. Then,
			$$
			f\perp_B g \iff \forall \mu\in \mathbb{T}\ \exists t\in \mathcal{M}_f \text{ such that } f(t)\perp_\mu g(t).
			$$
In the real case, we actually have
$$
			f\perp_B g \iff \exists t\in \mathcal{M}_f \text{ such that } f(t)\perp_B g(t).
$$
    \end{theorem}
	
The most technical part of the proof is contained in the next lemma, which is actually valid in $C_b(\Omega,Y)$. We still use the notation $\mathcal{M}_f$ for the (maybe empty) norm attainment set of a function $f\in C_b(\Omega,Y)$. 	

	\begin{lemma}\label{Topological Lemma}
Let $\Omega$ be a Hausdorff topological space, let $Y$ be a Banach space, and let $f\in C_b(\Omega, Y)$. Suppose that there exists a closed connected subset $D$ of $\Omega$ such that $D\subseteq\mathcal{M}_f$. Then, for every $g\in C_b(\Omega, Y)$, the set
			$$
			\{y^*(g(t))\colon t\in D,\ y^*\in S_{Y^*},\ y^*(f(t))=\|f\|\}
			$$
			is a connected subset of $\mathbb{C}$.
	\end{lemma}
	
	\begin{proof}
We follow the lines of the proof of the spatial numerical range of operators being connected given in \cite[Section 11]{BonsallDuncan1}. Consider the product $\Omega\times B_{Y^*}$, topologized by the product of the topology of $\Omega$ and the $w^*$ topology of $B_{Y^*}$. For any fixed $h\in C_b(\Omega, Y)$, define $\Theta_h\colon \Omega\times B_{Y^*}\longrightarrow \mathbb{C}$ by
			$$
			\Theta_h(t,y^*)=y^*(h(t)) \qquad \bigl((t,y^*)\in \Omega\times B_{Y^*}\bigr).
			$$
			Observe that
			\begin{align*}
				|\Theta_h(t,y^*)-\Theta_h(s,z^*)| & = |y^*(h(t))-z^*(h(s))|\\
				&\leq |y^*(h(t))-y^*(h(s))|+|J_Y(h(s))(y^*)-J_Y(h(s))(z^*)|\\
				& \leq \|h(t)-h(s)\|+|J_Y(h(s))(y^*)-J_Y(h(s))(z^*)|,
			\end{align*}
			where $J_Y\colon Y\longrightarrow Y^{**}$ denotes the canonical embedding. It follows from the continuity of $h$ and the $w^*$-continuity of $J_Y(h(s))$ that the map $\Theta_h$ is continuous.

Thus to prove our assertion, it is enough to show that
			$$
			\mathcal{A}=\{(t,y^*)\in D\times S_{Y^*}\colon y^*(f(t))=\|f\|\}
			$$
			is connected. Suppose by contradiction that $\mathcal{A}=F_1\cup F_2$, where $F_1,F_2$ are non-empty and closed in $\mathcal{A}$ with $\F_1 \cap F_2 =\emptyset$. The projections $\pi_1(F_1)$ and $\pi_1(F_2)$ are closed subsets of $\Omega$. Indeed, consider any net $(t_\tau)$ in $\pi_1(F_1)$ such that $t_\tau\to t_0$ in $\Omega$. Evidently, $\pi_1(F_1)\subseteq D$ and $D$ is closed. Thus, $t_0\in D$. For each $\tau$, consider $y_\tau^*\in S_{Y^*}$ such that $(t_\tau,y_\tau^*)\in F_1$. The net $(y^*_\tau)$ has a cluster point $y_0^*$ in $B_{Y^*}$, since $B_{Y^*}$ is $w^*$-compact. Thus, $(t_0,y_0^*)$ is a cluster point of the net $((t_\tau,y_\tau^*))$. Moreover, it follows from the continuity of $\Theta_f$ that $y_0^*(f(t_0))=\|f\|$. Thus, $y_0^*\in S_{Y^*}$ and we have $(t_0,y^*_0)\in \mathcal{A}$. Since $F_1$ is closed in $\mathcal{A}$, we have $(t_0,y_0^*)\in F_1$. Therefore, $t_0\in \pi_1(F_1)$ and $\pi_1(F_1)$ is a closed subset of $\Omega$. Similarly, $\pi_1(F_2)$ is a closed subset of $\Omega$. Note that $D=\pi_1(F_1)\cup \pi_1(F_2)$. It follows from the connectedness of $D$ that there exists $\widetilde{t}\in \pi_1(F_1)\cap \pi_1(F_2)$. Therefore, we may find $y_1^*$ and $y_2^*$ in $S_{Y^*}$ such that $(\widetilde{t},y_1^*)\in F_1$ and $(\widetilde{t},y_2^*)\in F_2$. Then,
			$$
			\mathcal{B}:=\left\{\left(\widetilde{t},(\lambda y_1^*+(1-\lambda)y_2^*)\right) \colon \lambda \in [0,1]\right\}
			$$
			is a connected subset and it is contained in $\mathcal{A}$. However, $(\mathcal{B}\cap F_1)$ and $(\mathcal{B}\cap F_2)$ are non-empty, closed in $\mathcal{B}$ and form a separation of $\mathcal{B}$. This contradicts the connectedness of $\mathcal{B}$.
	\end{proof}

We are now ready to give the pending proof of the theorem.
	
	\begin{proof}[Proof of Theorem~\ref{Compact continuous}]
We only prove the necessity as the sufficiency is straightforward. Suppose that $f\perp_B g$ and consider
			\begin{align*}
				& \mathcal{A}_1 := \{y^*(g(t))\colon t\in \mathcal{M}_f,\ y^*\in \ext(B_{Y^*}),\ y^*(f(t))=\|f\|\},\\
				& \mathcal{A}_2 := \{y^*(g(t))\colon t\in \mathcal{M}_f,\ y^*\in S_{Y^*},\ y^*(f(t))=\|f\|\}.
			\end{align*}
Observe that $\mathcal{A}_1\subseteq \mathcal{A}_2$ and that $0\in\conv(\mathcal{A}_1)$ by Corollary~\ref{corollary:C(K,Y)extreme}, hence $0\in\conv(\mathcal{A}_2)$.
Now, by Lemma~\ref{Topological Lemma}, $\mathcal{A}_2$ is connected. Therefore, by Lemma \ref{Convexity Lemma}, for every $\mu\in \mathbb{T}$ there exists $(t,y^*)\in \mathcal{M}_f\times S_{Y^*}$ such that $y^*(f(t))=\|f(t)\|=\|f\|$ and $\re \mu y^*(g(t))=0$. Hence, \eqref{equation:directionalortho} shows that $f(t)\perp_\mu g(t)$, as desired.
	\end{proof}

Our next aim is to apply Theorem~\ref{Compact continuous} to spaces of operators.
	Given Banach spaces $X$, $Y$ and $T\in \mathcal{L}(X,Y)$, let $M_T$ denote the (maybe empty) norm attainment set of $T$, that is,
	$$
	M_T:=\{x\in S_X\colon \|Tx\|=\|T\|\}.
	$$
In the real case, the result we get is the following one, which appeared in \cite{PaulSainGhosh-2016}.

	\begin{proposition}[\mbox{\cite[Theorem~2.1]{PaulSainGhosh-2016}}]\label{prop-BS-KXY-real}
\label{prop:BS-realcase}
		Let $X$ be a \emph{real} reflexive Banach space, let $Y$ be a \emph{real} Banach space, and let $T,A\in\mathcal{K}(X,Y)$. Suppose that $M_T=D\cup (-D)$ for a connected subset $D$ of $S_X$. Then, $T\perp_B A$ if and only if there exists $x\in D$ such that $Tx\perp_B Ax$.
	\end{proposition}

	\begin{proof}
	We only prove the necessity as sufficiency is obvious. Suppose that $T\perp_B A$ and consider
	\begin{align*}
		\mathcal{A}_1&:=\{y^*(Ax)\colon x\in \ext{B_X}, \ y^*\in \ext{B_{Y^*}}, \ y^*(Tx)=\|T\|\}; \\
		\mathcal{A}_2&:=\{y^*(Ax)\colon x\in D, \ y^*\in S_{Y^*}, \ y^*(Tx)=\|T\|\}.
	\end{align*}
	Let us show that $\mathcal{A}_1\subseteq \mathcal{A}_2$. Indeed,
	\begin{align*}
		\mathcal{A}_1 &\subseteq \{y^*(Ax)\colon x\in S_X, \ y^*\in S_{Y^*}, \ y^*(Tx)=\|T\|\} \\
		& = \{y^*(Ax)\colon x\in M_T, \ y^*\in S_{Y^*}, \ y^*(Tx)=\|T\|\}=\mathcal{A}_2.
	\end{align*}
	The first inclusion is obvious and the second equality is clear since $y^*(Tx)=\|T\|$ implies $x\in M_T$. For the third one, given $x\in M_T$, there exist $\theta \in \{-1,1\}$ and $z\in D$ with $x=\theta z$. If $y^*\in S_{Y^*}$ satisfies $y^*(Tx)=\|T\|$, then we have that
	$$
	(\theta y^*)(Tz)=y^*(Tx)=\|T\| \qquad \text{and} \qquad (\theta y^*)(Az)=y^*(Ax),
	$$
	and we deduce the desired equality. Now, $B_X$ equipped with the weak topology is a compact Hausdorff topological space. Consider the Banach space $C\bigl((B_X,w),Y\bigr)$. The identification $T\longmapsto \widetilde{T}$ where $\widetilde{T}=T|_{B_X}$, is an isometric embedding of $\mathcal{K}(X,Y)$ into  $C\bigl((B_X,w),Y\bigr)$. Thus, by virtue of this identification, we have that the set
	$$\mathcal{A}_3:=\{y^*(\widetilde{A}x)\colon x\in D, \ y^*\in S_{Y^*}, \ y^*(\widetilde{T}x)=\|\widetilde{T}\|\}$$
	coincides with $\mathcal{A}_2$ and is connected by Lemma~\ref{Topological Lemma}. It follows from Corollary~\ref{B-J-orth-operator-norm-compact} that $0\in\conv(\mathcal{A}_1)$ and so $0\in\conv(\mathcal{A}_3)$. Hence, Lemma~\ref{Convexity Lemma} gives that for every $\mu\in\{-1,1\}$ there exists $x_\mu\in D$ such that $\widetilde{T} x_\mu \perp_\mu \widetilde{A}x_\mu$. Therefore, there exists $x\in D$ such that $Tx\perp_B Ax$ as desired.
\end{proof}

The complex case can be treated similarly using the notion of directional orthogonality. Our main result extends \cite[Theorem~2.1]{PaulSainGhosh-2016} to the complex case and \cite[Theorem~7]{RoyBagchiSain} and \cite[Theorem~2.6]{PaulSainMalMandal} to the infinite-dimensional case. Observe that the connectedness of $M_T$ \emph{in the complex case} is equivalent to requiring that $M_T=\bigcup_{\theta\in\T} \theta D$ for a connected set $D$. In the real case, the second condition is weaker.

	\begin{theorem}\label{theorem:newBS-result-complex}
		Let $X$ be a \emph{complex} reflexive Banach space, let $Y$ be a \emph{complex} Banach space, and let $T,A\in\mathcal{K}(X,Y)$. Suppose that $M_T$ is connected. Then, $T\perp_B A$ if and only if for each $\gamma\in \T$ there exists $x\in M_T$ such that $Tx\perp_\gamma Ax$.
	\end{theorem}

This result can be proved following a completely analogous argument to the one for Proposition~\ref{prop-BS-KXY-real}, or alternatively, it can be established as a direct consequence of Theorem~\ref{Compact continuous} since $M_T$ is connected in this case.

When $X$ is finite-dimensional, the previous two results clearly apply.
	
	\begin{corollary}[\mbox{\cite[Theorem~7]{RoyBagchiSain} and \cite[Theorem~2.6]{PaulSainMalMandal}}]
		Let $X$ be a finite-dimensional space, let $Y$ be a Banach space, and let $T, A\in\mathcal{L}(X,Y)$. In the real case, suppose that $M_T=D\cup -D$ for a connected set $D$; in the complex case, suppose that $M_T$ is connected. Then, $T\perp_B A$ if and only if for each $\gamma\in \T$ there exists $x\in M_T$ such that $Tx\perp_\gamma Ax$.
	\end{corollary}
	
	We finally give another Bhatia-\v{S}emrl's type result which improves Corollary~\ref{corollary:uniformalgebras} for a class of holomorphic functions on the open unit disk.
	
	\begin{corollary}\label{corollary:holomorphic}
		Let $f\in \mathbb{A}(\mathbb{D})$ with $\|f\|=1$ and let $g$ be a holomorphic function on the open unit disk. Suppose that $M_f\subset \T$ is a connected subset of $\T$ and $g$ has radial limit $g^*$ with modulus one at every $z\in M_f$. Then,
		$$f \perp_B g \iff \forall \mu\in \mathbb{T}\ \exists z_0\in M_f\ \text{such that } \mu \overline{f(z_0)}g^*(z_0)\in \{i,-i\}.$$
	\end{corollary}
	
	\begin{proof}
		Observe that Corollary~\ref{corollary:uniformalgebras} gives that
		$$
		f\perp_B g \iff 0\in \conv\{\overline{f(z)}g^*(z)\colon z\in M_f\}.
		$$
		Using that the set $\{\overline{f(z)}g^*(z)\colon z\in M_f\}$ is connected and Lemma~\ref{Convexity Lemma}, we have that
		\begin{align*}
			f\perp_B g &\iff \forall \mu\in \T \ \exists z_0\in M_f \text{ such that } \re \mu \overline{f(z_0)}g^*(z_0)=0 
		\end{align*}
		and the result follows from $|f(z_0)|=|g^*(z_0)|=1$.
	\end{proof}
	
	This result applies to a class of inner functions of the disk algebra $\mathbb{A}(\mathbb{D})$, known as finite Blaschke products. A \emph{Blaschke product} of degree $n$ is defined by
	$$
	B_n(z):=z^k \prod_{j=1}^n \frac{|a_j|}{a_j} \frac{z-a_j}{1-\overline{a_j}z} \qquad (z\in \mathbb{D})
	$$
	where $k$ is an integer, $k\geq0$, and $0<|a_j|<1$, $1\leq j \leq n$. Observe that $|B_n(z)|=1$ for $z\in \T$. We refer the reader to \cite[page~310]{Rudin} for more information and background.

	\begin{example}\label{example:Blaschke}
	Let $B_m$ and $B_n$ be two Blaschke products of degree $m$ and $n$, respectively, viewed as elements of $\mathbb{A}(\mathbb{D})$. Then,
	$$
	B_n\perp_B B_m \Longleftrightarrow \forall \mu\in \mathbb{T}\ \exists z_0\in \mathbb{T}\ \text{such that } \mu \overline{B_n(z_0)}B_m(z_0)\in \{i,-i\}.
	$$
	\end{example}

	\section{Applications: obstructive results for spear vectors, spear operators, and Banach spaces with numerical index one}\label{section:An-application}
	
	The aim of this section is to use the results in Section~\ref{sect:B-J-orth-smooth-num-range} together with a mix of ideas from numerical ranges and BJ-orthogonality to obtain obstructive results for the existence of spear vectors, spear operators and, in particular, for the possibility of having $n(X)=1$ for a Banach space $X$. Let us introduce here some notation which will be used throughout this section. Let $Z$ be a Banach space. We write $\smooth(Z)$ to denote the set of smooth points of $Z$. For $z\in Z$, $z^\perp=\{x\in Z\colon z\perp_B x\}$ and $^\perp z=\{x\in Z\colon x\perp_B z\}$. Finally, $\StrExp(B_Z)$ denotes the set of \emph{strongly exposed points} of $B_Z$: $z_0\in \StrExp(B_Z)$ if there is $f_0\in S_{Z^*}$ such that whenever $\lim \re f_0(z_n)=1$ for $\{z_n\}\subset B_{Z}$, it follows that $\lim z_n=z_0$ in norm.

	\subsection{Spear vectors}
	Let us first give some notation. Let $Z$ be a Banach space and let $u\in S_Z$. The \emph{numerical radius} of $z\in Z$ with respect to $(Z,u)$ is
	$$
	v(Z,u,z):=\sup\{|\lambda| \colon \lambda\in V(Z,u,z)\}=
	\sup\{|\phi(z)| \colon \phi\in \F(B_{Z^*},u)\},
	$$
	which is a seminorm on $Z$ satisfying $v(Z,u,z)\leq \|z\|$ for every $z\in Z$.
	When $v(Z,u,\cdot)$ is a norm in $Z$, we say that $u$ is a \emph{vertex}.
	When $v(Z,u,z)=\|z\|$ for every $z\in Z$, $u$ is said to be a \emph{spear vector}. It is known that $u$ is a spear vector if and only if
	$$
	\max_{\theta\in \T}\|u+\theta z\|=1+\|z\|\quad \forall z\in Z.
	$$
	We write $\Spear(Z)$ for the set of spear vectors of $Z$.
	A lot of information on spear vectors can be found in Chapter 2 of the book \cite{KMMP-SpearsBook}.
	
	Consider a Banach space $Z$ and a vertex $u\in S_Z$, and let us consider $Z$ endowed with the norm $v_u$ given by the numerical radius with respect to $u$:
	$$
	v_u(z):=v(Z,u,z)=\sup\{|\phi(z)| \colon \phi\in \F(B_{Z^*},u)\} \qquad (z\in Z).
	$$
	Then, we can consider its dual space $(Z,v_u)^*$ consisting of the linear functionals $\psi\colon Z \longrightarrow \K$ satisfying
	$$
	\sup\{|\psi(y)| \colon y\in Z, \, v_u(y)\leq1\}<\infty,
	$$
	endowed with the norm
	$$
	v_u^*(\psi):=\sup\{|\psi(y)| \colon y\in Z, \, v_u(y)\leq1\} \qquad(\psi\in (Z,v_u)^*).
	$$
	For $x\in Z$ with $v_u(x)=1$, the numerical range of $y\in Z$ with respect to the numerical range space $\bigl((Z,v_u),x\bigr)$ is
	$$
	V\bigl((Z,v_u),x,y\bigr)=\{\psi(y)\colon \psi\in (Z,v_u)^*, \ v_u^*(\psi)=\psi(x)=1\}.
	$$
	
	Our obstructive result for spear vectors will follow from the next result.
	
	\begin{theorem}\label{theorem:smooth-not-orthogonal}
		Let $Z$ be a Banach space and let $u\in S_Z$ be a vertex of $Z$. If $z$ is smooth in $(Z,v_u)$, then $z\not\perp_B^{v_u}u$.
	\end{theorem}
	
	A technical part of the proof is contained in the following lemma which could be of independent interest.
	
	\begin{lemma}\label{lemma}
		Let $Z$ be a Banach space, let $u\in S_Z$ be a vertex, and let $z\in Z$ with $v_u(z)=1$. Then, $V\bigl((Z,v_u),z,u\bigr)\cap \T\neq \emptyset$.
	\end{lemma}
	
	\begin{proof}
		Since $v_u(z)=1$, there exist $\phi_0\in S_{Z^*}$ and $\theta_0\in \T$ such that $\phi_0(u)=\theta_0$ and $\phi_0(z)=1$. We claim that $\phi_0 \in (Z,v_u)^*$ and $v_u^*(\phi_0)=1$. Indeed,
		fix $y\in Z$ with $v_u(y)\leq 1$. As $\overline{\theta_0}\phi_0(u)=1$, we have that
		$\overline{\theta_0}\phi_0(y)\in V(Z,u,y)$, hence $|\phi_0(y)|\leq v(Z,u,y)=v_u(y)\leq 1$. This shows that $\phi_0\in (Z,v_u)^*$ and
		$$
		v_u^*(\phi_0)=\sup\{|\phi_0(y)| \colon y\in Z, \, v_u(y)\leq1\}\leq 1.
		$$
		On the other hand, since $v_u(u)=1$, we have that $v_u^*(\phi_0)\geq |\phi_0(u)|=1$.
		
		This, together with $\phi_0(z)=1$, gives that
		\begin{equation*}
			\theta_0=\phi_0(u)\in V\bigl((Z,v_u),z,u\bigr)=\{\psi(u)\colon \psi\in (Z,v_u)^*, \ v_u^*(\psi)=\psi(z)=1\}. \qedhere
		\end{equation*}
	\end{proof}
	
	We are now ready to present the pending proof.
	
	\begin{proof}[Proof of Theorem~\ref{theorem:smooth-not-orthogonal}]
		As $z$ is a smooth point, we have that $z\neq 0$ so, since $u$ is a vertex, this implies that $v_u(z)\neq 0$. Now, $V\left((Z,v_u),\dfrac{z}{v_u(z)},u\right)$ is a singleton set by Lemma~\ref{lemma:smooth-V} as the norm of $(Z,v_u)$ is smooth at $z$, hence also at $\dfrac{z}{v_u(z)}$. Moreover, since $v_u\left(\dfrac{z}{v_u(z)}\right)=1$, it follows from Lemma~\ref{lemma} that
		$$
		V\left((Z,v_u),\dfrac{z}{v_u(z)},u\right)=\{\theta_0\}
		$$
		for some $\theta_0\in \T$. Hence, $0\notin V\left((Z,v_u),\dfrac{z}{v_u(z)},u\right)$ so Proposition~\ref{prop:BJ-using-numranges} gives that $\dfrac{z}{v_u(z)}\not\perp_B^{v_u}u$ and hence $z\not\perp_B^{v_u}u$.
	\end{proof}
	
	We are ready to obtain the promised obstructive result for spear vectors.
	
	\begin{corollary}\label{cor:smooth-orth-spear}
		Let $Z$ be a Banach space and $u\in S_Z$. If there exists a smooth point $z_0$ in $Z$ such that $z_0\perp_B u$, then $(Z,v_u)$ is not isometrically isomorphic to $Z$. In particular,
		$u$ is not a spear vector or, in other words,
		$$
		\left(\bigcup_{z\in \smooth(Z)} z^\perp\right)\cap \Spear(Z)=\emptyset
		\quad \text{and} \quad
		\smooth(Z)\cap \left(\bigcup_{z \in \Spear(Z)} {^\perp z}\right)=\emptyset.
		$$
	\end{corollary}
	
	\begin{proof}
		Suppose on the contrary that $(Z,v_u)$ is isometrically isomorphic to $Z$. Since $z_0$ is smooth in $Z$ and $z_0\perp_B u$, we have that $z_0$ is smooth in $(Z,v_u)$ and $z_0\perp_B^{v_u} u$, which contradicts
		Theorem~\ref{theorem:smooth-not-orthogonal}. If $u$ is a spear vector, then the identity map $\Id\colon(Z,\|\cdot\|)\longrightarrow (Z,v_u)$ is an isometric isomorphism.
	\end{proof}
	
	Observe that the last part of Corollary~\ref{cor:smooth-orth-spear} can be shown with a geometrical argument. Indeed, if $z_0\in S_Z$ is a smooth point in $Z$, then there exists a unique $\phi\in S_{Z^*}$ such that $\phi(z_0)=1$ and $\phi\in \ext(B_{Z^*})$. Now, if $u\in S_{Z}$ is a spear vector, then $|\phi(u)|=1$ by \cite[Corollary 2.8]{KMMP-SpearsBook}, and so $z_0$ cannot be BJ-orthogonal to $u$. Note that this argument only depends on the subspace generated by $\{u,z_0\}$.
	\begin{remark}
		If $z_0$ is a smooth point in a Banach space $Z$ and $u\in S_Z$ is a spear vector in $\spn\{z_0,u\}$, then $z_0$ is not BJ-orthogonal to $u$. 
	\end{remark}
	
	\subsection{Spear operators}\label{subsec:obstructive-spearoperators}
	In the case when $Z=\mathcal{L}(X,Y)$ for some Banach spaces $X$ and $Y$, spear vectors are called \emph{spear operators}, which were introduced in \cite{Ardalani} and have been deeply studied in \cite{KMMP-SpearsBook}, where we refer for more information and background.
	
	Our aim here is to particularize Corollary~\ref{cor:smooth-orth-spear} for the numerical radius with respect to an operator and for spear operators. The results follow directly from the above ones, but we include some particular notation for this case. Given Banach spaces $X$ and $Y$, and $G\in \mathcal{L}(X,Y)$ with $\|G\|=1$, the \emph{numerical radius of $T\in \mathcal{L}(X,Y)$ with respect to $G$} is
	\begin{align*}
		v_G(T)  :=v\bigl(\mathcal{L}(X,Y),G,T \bigr)  &=\inf_{\delta>0}\sup\bigl\{|y^*(Tx)|\colon y^*\in S_{Y^*},\,x\in S_X,\, \re y^*(Gx)>1-\delta\bigr\} \\
		& = \sup \bigl\{\lim |y_n^*(T x_n)|\colon \{y_n^*\}\subset S_{Y^*},\,\{x_n\}\subset S_X,\, \lim y_n^*(G x_n)=1 \bigr\},
	\end{align*}
	where the second and third equalities hold by \cite[Proposition~2.14]{KMMPQ} and our Theorem~\ref{theorem:num-range-C}, respectively. We refer to \cite{KMMPQ} for background on numerical radius with respect to an operator.
	
	The main result of the previous subsection in this setting reads as follows.

	\begin{corollary}\label{cor:smooth-orth-spear-L(XY)}
		Let $X$, $Y$ be Banach spaces and let $G\in \mathcal{L}(X,Y)$ with $\|G\|=1$. If there exists a smooth operator $T$ in $\mathcal{L}(X,Y)$ such that $T \perp_B G$, then $\bigl(\mathcal{L}(X,Y),v_G\bigr)$ is not isometrically isomorphic to $\mathcal{L}(X,Y)$. In particular, $G$ is not a spear operator or, in other words,
		$$
		\left(\bigcup_{T\in \smooth(\mathcal{L}(X,Y))} T^\perp\right)\cap \Spear(\mathcal{L}(X,Y))=\emptyset
		\ \ \text{and} \ \
		\smooth(\mathcal{L}(X,Y))\cap \left(\bigcup_{G \in \Spear(\mathcal{L}(X,Y))} {^\perp G}\right)=\emptyset.
		$$
	\end{corollary}
	
	Our next aim is to provide an obstructive result for the existence of spear operators which uses the geometry of the domain and range spaces instead of the geometry of the space of operators and so it would be easier to apply. Other restrictions on the geometry of the domain and range spaces to the existence of spear operators can be found in \cite[Ch.~6]{KMMP-SpearsBook}. To obtain the desired result through the application of Corollary~\ref{cor:smooth-orth-spear-L(XY)}, we present the following lemma, which provides a tool to construct smooth operators which are BJ-orthogonal to a given one, under mild assumptions.
	
	\begin{lemma}\label{lemma:existence-operator-smooth-orth}
		Let $X$, $Y$ be Banach spaces and suppose that $x_0\in B_X$ is strongly exposed by $x_0^*\in S_{X^*}$. Given $A\in\mathcal{L}(X,Y)$, suppose that there is a smooth point $u_0\in Y$ satisfying $u_0\perp_B Ax_0$. Then, the operator $T\in\mathcal{L}(X,Y)$ given by $T(x)=x_0^*(x)u_0$ is smooth and satisfies $T\perp_B A$.
	\end{lemma}
	
	\begin{proof}
		Observe that $T$ clearly satisfies the hypotheses of Proposition~\ref{prop:suff-condition-smooth-operator} so it is a smooth operator. Besides, using that $u_0\perp_B Ax_0$, we have that
		$$
		\|T+\lambda A\| \geq \|Tx_0+\lambda Ax_0\|=\|u_0+\lambda Ax_0\|\geq \|u_0\|=\|T\|
		$$
		for every $\lambda\in\K$. Consequently, $T\perp_B A$. 	
	\end{proof}
	
	We obtain the desired result as an immediate consequence of the previous lemma and Corollary~\ref{cor:smooth-orth-spear-L(XY)}.
	
	\begin{corollary}\label{corollary:main-consequence-spearoperators}
		Let $X$, $Y$ be  Banach spaces and let $G\in \mathcal{L}(X,Y)$ with $\|G\|=1$. Suppose that
		there are $x_0\in \StrExp(B_X)$ and $u_0 \in \smooth(Y)$ satisfying that $u_0\perp_B Gx_0$. Then, $\bigl(\mathcal{L}(X,Y),v_G\bigr)$ is not isometrically isomorphic to $\mathcal{L}(X,Y)$. In particular, $G$ is not a spear operator. 
	\end{corollary}
	
	The previous result can be reformulated in a more suggestive manner as follows.
	
	\begin{corollary}
		Let $X$, $Y$ be  Banach spaces and let $G\in \mathcal{L}(X,Y)$ with $\|G\|=1$ be a spear operator. Then,
		$$
		\left(\bigcup_{y\in \smooth(Y)} y^\perp\right)\cap G(\StrExp(B_X))=\emptyset
		\quad \text{and} \quad
		\smooth(Y)\cap \left(\bigcup_{x\in \StrExp(B_X)} {^\perp (Gx)}\right)=\emptyset.
		$$
	\end{corollary}
	
	We obtain an interesting result when $X$ contains strongly exposed points and $Y$ is a smooth Banach space with dimension at least two.
	
	\begin{corollary}\label{corollary:strexpnonempty-smooth-smooth-ortogona}
		Let $X$ be a Banach space with $\StrExp(B_X)\neq \emptyset$ and let $Y$ be a smooth Banach space of dimension at least two. Then, for every $A\in\mathcal{L}(X,Y)$ there is a smooth  operator $T\in\mathcal{L}(X,Y)$ satisfying $T\perp_B A$. Consequently, there are no spear operators in $\mathcal{L}(X,Y)$.
	\end{corollary}
	
	Observe that the last assertion of this result extends \cite[Proposition~6.5.a]{KMMP-SpearsBook} when $X$ contains strongly exposed points (in particular, when $X$ has the RNP) and provides a partial answer to \cite[Problem~9.12]{KMMP-SpearsBook}.
	
	\begin{proof}[Proof of Corollary~\ref{corollary:strexpnonempty-smooth-smooth-ortogona}]
		To prove the first assertion, take $x_0\in \StrExp(B_X)$, then $Ax_0\in Y$ and, since $\dim(Y)\geq2$, there exists $u_0\neq 0$ smooth point of $Y$ such that $u_0\perp_B Ax_0$. Now, Lemma~\ref{lemma:existence-operator-smooth-orth} gives the existence of a smooth operator $T$ BJ-orthogonal to $A$. The second assertion follows from Corollary~\ref{cor:smooth-orth-spear-L(XY)}.
	\end{proof}

	\subsection{Banach spaces with numerical index one} \label{subsec:banachspacesnumericalindexone}
	We finally particularize the results of the previous subsection to the case when $X=Y$ and $G=\Id_X$. In this case, we use the usual notation $v(\cdot)$ for the numerical radius (instead of $v_{\Id}$) which was introduced in Subsection~\ref{subsection:numericalradiusnorm}. We need the following notation.
	The \emph{numerical index} of a Banach space $X$ is defined by
	$$
	n(X):=\inf\{v(T)\colon T\in S_{\mathcal{L}(X)}\}.
	$$
	Equivalently, $n(X)$ is the greatest constant $k\geq 0$ such that $k\|T\|\leq v(T)$ for every $T\in\mathcal{L}(X)$. Note that $0\leq n(X)\leq 1$ and $n(X)>0$ if and only if $v(\cdot)$ and $\|\cdot\|$ are equivalent norms on $\mathcal{L}(X)$. The case $n(X)=1$ is equivalent to the fact that $\Id_X$ is a spear operator and we say that $X$ is a \emph{Banach space with numerical index one} or that $X$ has \emph{numerical index one}. We refer the reader to the expository paper \cite{KMP} and to Chapter~1 of the already cited book \cite{KMMP-SpearsBook} for an overview of classical and recent results on Banach spaces with numerical index one. Let us mention that some isomorphic and isometric restrictions on a Banach space $X$ to have numerical index one are known: $X^*$ cannot be smooth nor strictly convex \cite[Theorem~2.1]{ConvexSmooth} and, in the \emph{real} infinite-dimensional case, $X^*$ contains a copy of $\ell_1$ \cite[Corollary~4.9]{SCDsets}. It is open, as far as we know, whether the latter result extends to the complex case and whether a Banach space with numerical index one can be smooth or strictly convex (\cite{ConvexSmooth} or \cite[Problem~9.12]{KMMP-SpearsBook}). The particularization of the results of the previous subsection to the case of the identity reads as follows.
	
	\begin{corollary}\label{cor:X-exp-point-n(X)<1}
		Let $X$ be a Banach space. If there are $x_0\in \StrExp(B_X)$ and $u_0\in \smooth(X)$ such that $u_0\perp_B x_0$, then $X$ does not have numerical index one.
	\end{corollary}
	
	This result can be written in the following more suggestive way:
	
	\begin{corollary}
		Let $X$ be a Banach space with numerical index one. Then,
		$$
		\left(\bigcup_{x\in \smooth(X)} x^\perp\right)\cap \StrExp(B_X)=\emptyset
		\quad \text{and} \quad
		\smooth(X)\cap \left(\bigcup_{x\in \StrExp(B_X)} {^\perp x}\right)=\emptyset.
		$$
	\end{corollary}
	
	The above result provides a necessary condition to have numerical index one for a Banach space in the way that was asked in \cite[Problem 11]{KMP}: {\slshape Find necessary and sufficient conditions for a Banach space to have numerical index one which do not involve operators.}
	
	The next is a consequence of Corollary~\ref{cor:X-exp-point-n(X)<1} which gives a partial answer to the question of whether there is a smooth Banach space with numerical index one.
	
	\begin{corollary}\label{corollary:smooth-stronglyexposednotempty-nonX=1}
		Let $X$ be a smooth Banach space of dimension at least two such that $\StrExp(B_X)\neq \emptyset$. Then, $X$ does not have numerical index one.
	\end{corollary}
	
	This applies, in particular, when $X$ has the RNP.
	
	\begin{corollary}
		Let $X$ be a smooth Banach space of dimension at least two having the RNP. Then, $X$ does not have numerical index one.
	\end{corollary}
	
	\vspace{2ex}
	
	\noindent	\textbf{Acknowledgements.\ }
	The authors would like to thank Rafael Pay\'{a} for kindly answering several inquires on the topics of this paper. We also thank Abraham Rueda for some interesting remarks on the topics.

\end{document}